\documentclass[12pt]{amsart}
\usepackage{bm}
\usepackage[all]{xypic}
\usepackage{amsmath}
\usepackage{amscd}
\usepackage{amssymb}
\usepackage{enumerate}
\usepackage{amsfonts}
\usepackage{graphicx}
\usepackage[all]{xy}
\usepackage{mathrsfs}
\usepackage[active]{srcltx}
\usepackage{ftnxtra}
\usepackage{hyperref}
\usepackage{bm}

\DeclareFontEncoding{OT2}{}{}

\numberwithin{equation}{section}

\newcommand{\M}{\operatorname{M}}

\newcommand{\rar}{\longrightarrow}

\newcommand{\rarab}[1]{\overset{#1}{\longrightarrow}}

\newcommand{\al}{\alpha}
\newcommand{\be}{\beta}
\newcommand{\ga}{\gamma}
\newcommand{\Ga}{\Gamma}
\newcommand{\de}{\delta}

\newcommand{\la}{\lambda}
\newcommand{\La}{\Lambda}

\newcommand{\eps}{\epsilon}
\newcommand{\sg}{\sigma}

\newcommand{\om}{\omega}
\newcommand{\Om}{\Omega}

\newcommand{\bF}{{\mathbb F}}

\newcommand{\bN}{{\mathbb N}}

\newcommand{\bQ}{{\mathbb Q}}
\newcommand{\bR}{{\mathbb R}}
\newcommand{\bS}{{\mathbb S}}
\newcommand{\bT}{{\mathbb T}}

\newcommand{\bZ}{{\mathbb Z}}

\newcommand{\cA}{{\mathcal A}}

\newcommand{\cG}{{\mathcal G}}

\newcommand{\cO}{{\mathcal O}}
\newcommand{\cP}{{\mathcal P}}

\newcommand{\caR}{{\mathcal R}}
\newcommand{\cS}{{\mathcal S}}

\newcommand{\cX}{{\mathcal X}}

\newcommand{\fa}{{\mathfrak a}}
\newcommand{\fb}{{\mathfrak b}}

\newcommand{\fp}{{\mathfrak p}}
\newcommand{\fq}{{\mathfrak q}}

\newcommand{\abs}[1]{\vert #1\vert}

\newcommand{\End}{\operatorname{End}}

\newcommand{\Hom}{\operatorname{Hom}}

\newcommand{\Aut}{\operatorname{Aut}}
\newcommand{\Span}{\operatorname{Span}}

\newcommand{\Isom}{\operatorname{Isom}}

\newcommand{\id}{\operatorname{id}}

\newcommand{\diag}{\operatorname{diag}}

\newcommand{\Gal}{\operatorname{Gal}}
\newcommand{\GL}{\operatorname{GL}}

\newcommand{\val}{\operatorname{val}}
\newcommand{\rad}{\operatorname{rad}}

\newcommand{\tr}{\operatorname{Tr}}

\newcommand{\sbr}{\smallbreak}

\newcommand{\indlim}[1]{\lim\limits_{\underset{N}{\longrightarrow}}}
\newtheorem{thm}{Theorem}[section]
\newtheorem{cor}[thm]{Corollary}
\newtheorem{lem}[thm]{Lemma}

\newtheorem{prop}[thm]{Proposition}
\newtheorem*{t:mainth0}{Theorem}

\theoremstyle{remark}
\newtheorem{rem}[thm]{Remark}

\newtheorem{prob}{Problem}
\newtheorem{question}{Question}
\newtheorem{example}{Example}
\newtheorem{case}{Case}

\newcommand{\bbe}{\begin{equation}}
\newcommand{\ee}{\end{equation}}

\title[Computability of dimension groups]{Computability of dimension groups}

\author{Maria Sabitova}
\address{Department of Mathematics\\
CUNY Queens College\\ 65-30 Kissena Blvd.\\ Flushing, NY
11367\\USA} 
\email{Maria.Sabitova@qc.cuny.edu}

\date{\today}

\begin{document}

\tableofcontents

\begin{abstract}  
We investigate the computability of  
the isomorphism set $\operatorname{Iso}(G_A,G_B)$ between $G_A$ and $G_B$, where $G_A$ is a subgroup of $\bQ^n$ generated by columns of integer powers of a non-singular $n \times n$-matrix $A$ with integer entries. 
Assuming that the characteristic polynomial of $A$ is irreducible -- and under an additional condition when $n$ is not prime -- we prove that  
$\operatorname{Iso}(G_A,G_B)$ is computable; that is, there exists an algorithm that determines the structure in finitely many steps. We also present illustrative examples.
%
%
 \end{abstract}
 


\maketitle

\section{Introduction}
It is well known that subgroups of $\bQ$ have a complete classification \cite{b}. However, the analogous problem for subgroups of 
$\bQ^n$ -- equivalently, torsion-free abelian groups of rank at most $n$ -- remains open 
(even in light of \cite{f2}). As shown in \cite{thom}, the classification problem becomes increasingly complex as $n$ grows. In our previous works 
\cite{s1}, \cite{s2}, \cite{ms}, \cite{s3}, we focus on a special class of subgroups of $\bQ^n$: those defined by integer matrices. These groups arise naturally in the study of toroidal solenoids and $\bZ^n$-odometers.

\sbr

Given a non-singular matrix $A\in\operatorname{M}_n(\bZ)$, the associated toroidal solenoid  is an $n$-dimensional topological abelian group. Such solenoids were first introduced by M.~C.~McCord in 1965 \cite{m}. The earliest examples, corresponding to $n=1$ and $A=2$, were studied by L.~Vietoris in 1927 \cite{v} and later extended by van Dantzig in 1930 to arbitrary $A\in\bZ$ \cite{d}. The same class of subgroups also arises as the first cohomology groups of constant base $\bZ^n$-odometers \cite{gps}, and they appear in the context of symbolic dynamics and dimension group theory via subshifts of finite type \cite{bs}.

\sbr

More precisely, for a non-singular matrix $A\in\M_n(\bZ)$, we define the subgroup $G_A$ of $\bQ^n$ generated by columns of powers $A^i$, $i\in\bZ$. 
In this paper, we investigate computability questions related to these subgroups $G_A$. In \cite{s1} and \cite{s2}, we addressed the classification problem for these groups: given two non-singular matrices $A,B\in \operatorname{M}_n(\bZ)$, we determined when the associated groups $G_A$ and $G_B$ are isomorphic (as abstract groups). The case $n=2$ was treated in \cite{s1}, while the general case was handled in \cite{s2}. Applications to $\bZ^n$-odometers were explored in \cite{ms}.
In \cite{s3}, we studied the endomorphism ring $\End(G_A)$ for arbitrary $n$ and provided a general criterion for when a rational matrix $T\in\M_n(\bQ)$ defines an endomorphism of $G_A$. When the characteristic polynomial of $A$ is irreducible -- and under an additional assumption if $n$ is not prime -- we showed that every such endomorphism commutes with $A$, the ring 
$\End(G_A)$ is commutative and a finitely generated module over the ring $\bZ[t,t^{-1}]$, and the automorphism group $\Aut(G_A)$ is a finitely generated abelian group.
Building on these theoretical results, the present paper aims to make the methods more practical by developing an algorithmic framework for computing examples. Although various examples were computed in our earlier works, here we adopt a more systematic computational approach. In particular, when the characteristic polynomial of $A$ is irreducible, so that the structure of $G_A$ is especially rigid, we prove that the isomorphism set between $G_A$ and $G_B$ is computable. That is, there exists an algorithm that determines the structure in finitely many steps. All components involved in the proofs are themselves computable and are implemented in mathematical software such as \cite{sage}. We use \cite{sage} together with the database \cite{db} to test our proofs on both two- and higher-dimensional examples. At this stage, we do not address the computational complexity or efficiency of these procedures.
%
%

\sbr

Our work is inspired by that of K.~H.~Kim and F.~W.~Roush \cite{kr}, who investigated the decidability of epimorphisms between dimension groups. In their setting, $G_A$ appears as a special case of a dimension group with $A$ non-singular (whereas in \cite{kr}, singular matrices are also allowed). Notably, epimorphisms between such groups are more restricted than general homomorphisms: if there exists an epimorphism from $G_A$ to $G_B$, then $A$ and $B$ must be conjugate in $\GL_n(\bQ)$. This implication does not hold for arbitrary group homomorphisms. For example, there exist isomorphic groups $G_A\cong G_B$ for which  matrices $A$ 
and $B$ have different eigenvalues and hence are not conjugated by a matrix in $\GL_n(\bQ)$ \cite[Example 7.6]{s1}. 
In particular, the results of \cite{kr} apply to the endomorphism ring $\End(G_A)$ and the automorphism group 
$\Aut(G_A)$ under additional assumptions;  we present new algorithms and examples for computing each.  
We also take inspiration from \cite{cp}, where the authors investigate computability of the linear representation group of constant-base 
$\bZ^2$-odometer systems.

\section{Notation}
\noindent $A,B\in\M_n(\bZ)$ non-singular \\
$h_A\in\bZ[x]$ characteristic polynomial of $A$ \\
$G_A=\left\{A^{k}{\bf x}\, \vert \, {\bf x}\in\bZ^n,\,k\in\bZ\right\}$ \\
$\caR=\bZ\left[\frac{1}{\det A} \right]$ \\
$\cP=\cP(A)=\{\text{primes }p\in\bN\text{ dividing }\det A\}$ \\
$\cP'=\cP'(A)=\left\{p\in\cP\,\vert\,\,h_A\not\equiv x^n\,(\text{mod }p)\right\}$ \\
$t_p=$ multiplicity of zero in the reduction of ${h}_A$ modulo $p$ \\
$\val_px$, $x\in\bN$, the highest power of $p$ dividing $x$ \\ 
$\abs{x}_p=p^{-\val_px}$ \\
$\bQ_p=$ field of $p$-adic numbers \\
$\bZ_{(p)}=\left.\left\{\frac{m}{n}\in\bQ \,\, \right\vert \, m,n\in\bZ,\,\,n\ne 0,\,\,(p,n)=1\right\}$ \\
$\bZ_p=$ ring of $p$-adic integers \\
$\bQ_p^n=(\bQ_p)^n=\bQ_p\times\cdots\times\bQ_p$ \\
$\bZ_p^n=(\bZ_p)^n=\bZ_p\times\cdots\times\bZ_p$ \\
$\overline{G}_{A,p}=G_A\otimes_{\bZ}\bZ_p$ \\
$\overline{\bQ}=$ algebraic closure of $\bQ$ \\
$\la=$ eigenvalue of $A$ \\
$\mu=$ eigenvalue of $B$ \\
$K=\bQ(\la)$ \\ 
$l=[\cO_K:\bZ[\la]]$ \\
${\bf u}=\left(\begin{matrix} u_1 & \ldots & u_n \end{matrix}\right)^t$ eigenvector of $A$ corresponding to $\la$ \\
${\bf v}=\left(\begin{matrix} v_1 & \ldots & v_n \end{matrix}\right)^t$ eigenvector of $B$ corresponding to $\mu$ \\
$\bZ[{\bf u}]=\{m_1u_1+\cdots +m_nu_n\,\vert\,m_1,\ldots,m_n\in\bZ \}$ \\
$\bZ[{\bf v}]=\{m_1v_1+\cdots +m_nv_n\,\vert\,m_1,\ldots,m_n\in\bZ \}$ \\
$\{\la_1,\ldots,\la_n\}=$ eigenvalues of $A$ \\
$\{\mu_1,\ldots,\mu_n\}=$ eigenvalues of $B$ \\
$\{\sg_1=\id,\sg_2,\ldots,\sg_n\}=$ embeddings of $K$ into $\overline{\bQ}$ \\
$M=\left(\begin{matrix} \sg_1({\bf u}) & \ldots & \sg_n({\bf u}) \end{matrix}\right)\in\M_n(\overline{\bQ})$ \\
$N=\left(\begin{matrix} \sg_1({\bf v}) & \ldots & \sg_n({\bf v}) \end{matrix}\right)\in\M_n(\overline{\bQ})$ \\
$\La=\diag\left(\begin{matrix} \sg_1(\la)& \ldots & \sg_n(\la)\end{matrix}\right)$ \\ 
$\Ga=\diag\left(\begin{matrix} \sg_1(\mu)& \ldots & \sg_n(\mu)\end{matrix}\right)$ \\ 
$m=(\det M)^2\in\bZ$ \\
$\cO_K=$ ring of integers of $K$ \\
$\cO_K^{\times}=$   units of $\cO_K$ \\
$\fp=$ prime ideal of $\cO_K$ above $p$ \\
$\val_{\fp}(x)=$ $\fp$-adic valuation of $x\in K$ \\
$\cO_{K,\la}=\{x\in K\,\vert\, \val_{\fp}(x)\geq 0\text{ for any prime ideal }\fp\text{ of }\cO_K\text{ not dividing }\la\}$ \\ 
$\cO^{\times}_{K,\la}=\{x\in K-\{0\}\,\vert\, \val_{\fp}(x)=0\text{ for any prime ideal }\fp\text{ of }\cO_K\text{ not dividing }\la\}$ \\
$K_{\fp}=$ completion of $K$ with respect to $\fp$ \\
$\cO_{\fp}=$ ring of integers of $K_{\fp}$ \\
$\rad(n)=$ product of all distinct prime divisors of $n\in\bZ$ \\
$(u,v)=$ greatest common divisor of $u,v\in\bZ$ \\
$X_A=$ $\bZ^n$-odometer defined by $A$ \eqref{eq:xg} \\
$\vec{N}(X_A)=$ linear representation group of $X_A$ \eqref{eq:linrepgr} \\

\section{Procedures to compute $\End(G_A)$ and $\Aut(G_A)$}
Let $A\in \operatorname{M}_n(\bZ)$ be a non-singular $n\times n$-matrix with integer entries. Denote
\bbe\label{eq:g}
G_A=\left\{A^{k}{\bf x}\, \vert \, {\bf x}\in\bZ^n,\,k\in\bZ\right\},\quad \bZ^n\subseteq G_A\subseteq\bQ^n,
\ee
where ${\bf x}\in\bZ^n$ is written as a column. Then $G_A$ is a subgroup of $\bQ^n$ (under addition). 
Denote by $\End(G_A)$ 
the endomorphism ring of $G_A$, consisting of all (group) homomorphisms $T:G_A
\rar G_A$. Denote by $\Aut(G_A)$ the group of all (group) isomorphisms $T:G_A
\rar G_A$. 
In this section, we address the computability of $\End(G_A)$ and $\Aut(G_A)$, {\it i.e.}, if there exists an algorithm that given a non-singular $A\in\M_n(\bZ)$ computes the endomorphism ring $\End(G_A)$ (resp., the group $\Aut(G_A)$) in finitely many steps. Denote 
\bbe\label{eq:car}
\caR=\caR(A)=\bZ\left[\frac{1}{\det A} \right]=\left\{k(\det A)^l\,\,\Big\vert\,\,k,l\in\bZ  \right\}.
\ee
If $T\in\End(G_{A})$, then $T\in\M_n(\caR)$. Indeed, one can check that any group homomorphism from $G_A$ to 
$G_A$ is induced by multiplication by a matrix $T\in\M_n(\bQ)$. Moreover, from the definition of $G_A$, there exists $i\in\bN\cup\{0\}$ such that $A^iT\in\M_n(\bZ)$, hence the entries of $T$ are elements of $\caR$, {\it i.e.}, $T\in\M_n(\caR)$. Thus, 
$\End(G_A)\subseteq \M_n(\caR)$ and $\Aut(G_A)\subseteq \GL_n(\caR)$, where $\GL_n(\caR)$ denotes the group of non-singular matrices
$T\in\M_n(\caR)$ such that $T^{-1}\in\M_n(\caR)$. 
Let 
\begin{eqnarray*}
\cP&=&\cP(A)=\{\text{primes }p\in\bN\text{ dividing }\det A\}, \\
\cP'&=&\cP'(A)=\left\{p\in\cP\,\vert\,h_A\not\equiv t^n\,(\text{mod }p)\right\},
\end{eqnarray*}
where $h_A\in\bZ[t]$ is the characteristic polynomial of $A$. First, we look at the easy cases when $\cP(A)=\emptyset$ or
$\cP'(A)=\emptyset$.

\begin{lem}[\cite{s3}]\label{l:easy0}
\begin{enumerate}[$(1)$]
\item If $\cP(A)=\emptyset$,  equivalently, $A\in\GL_n(\bZ)$, then 
$$G_A=\bZ^n,\quad \End({G}_A)=\M_n(\bZ),\quad \Aut({G}_A)=\GL_n(\bZ).$$
Also, for a non-singular $B\in\M_n(\bZ)$ we have that $G_A\cong G_B$ if and only if 
$B\in\GL_n(\bZ)$ if and only if $G_A=G_B$. \\

\item If $\cP'(A)=\emptyset$ and $A\not\in\GL_n(\bZ)$, then $$\End({G}_A)=\M_n(\caR),\quad \Aut({G}_A)=\GL_n(\caR).$$ 
Also, for a non-singular $B\in\M_n(\bZ)$ we have that $G_A\cong G_B$ if and only if 
$\det A$, $\det B$ have the same prime divisors and $\cP'(B)=\emptyset$.
\end{enumerate}
\end{lem}
Clearly, the conditions in Lemma \ref{l:easy0} are computable. 

\begin{prop}\label{pr:decide}
Let $A\in\M_n(\bZ)$ be an arbitrary non-singular matrix. Then elements of $G_A$, $\End(G_A)$, and $\Aut(G_A)$ 
are computable. Also, given non-singular $A,B\in\M_n(\bZ)$ and $T\in\M_n(\bQ)$ $($resp., $T\in\GL_n(\bQ))$, it is decidable whether $T$ is a homomorphism $($resp., isomorphism$)$ from $G_A$ to $G_B$. 
\end{prop}

\begin{proof}
We begin by showing that the elements of   $G_A$ are computable. In other words, given a vector ${\bf x}\in\bQ^n$ and a non-singular $A\in\M_n(\bZ)$, one can write an algorithm that decides whether ${\bf x}\in G_A$ in finitely many steps. Indeed, 
${\bf x}\in G_A$ if and only if there exists $i\in\bN\cup\{0\}$ such that $A^i{\bf x}\in\bZ^n$. 
Let $k\in\bN$ be such that $k{\bf x}\in \bZ^n$, {\it e.g.}, the least common multiple of the denominators of the entries of ${\bf x}$. The set $\{A^j\}_{j\in\bN\cup\{0\}}$ is finite modulo $k\cdot\M_n(\bZ)$. Thus, it is enough to check
whether $A^i{\bf x}\in\bZ^n$ for finitely many powers $i$.

\sbr

Let $A,B\in\M_n(\bZ)$ be non-singular and let $T\in\M_n(\bQ)$ be non-zero. We show that it is decidable whether $T$ is a homomorphism
from $G_A$ to $G_B$, {\it i.e.}, one can write an algorithm that decides whether $T(G_A)\subseteq G_B$ in finitely many steps. 
It follows from the definition of $G_A$ that if $T(G_A)\subseteq G_B$, then
there exists $i\in\bN\cup\{0\}$ such that $B^iT\in\M_n(\bZ)$. Equivalently, each column of $T$ lies in $G_B$, which, by the previous paragraph, is a decidable condition. Moreover, if this holds, then one can compute a power $i$ such that $B^iT\in\M_n(\bZ)$. Then, 
  $T(G_A)\subseteq G_B$ if and only if $B^iT(G_A)\subseteq G_B$. Thus, 
without loss of generality, $T\in\M_n(\bZ)$. If $A\in\GL_n(\bZ)$, then $G_A=\bZ^n$ and $T$ is a homomorphism. Hence, for the rest of the proof we assume that $\cP(A)\ne\emptyset$. 
%
%
%
%
Let $\overline{\bQ}$ be a fixed algebraic closure of $\bQ$ and let
$L\subset \overline{\bQ}$ denote a number field, containing the eigenvalues of both $A$ and $B$. 
We find Jordan forms  
$A=MJM^{-1}$, $B=N\tilde{J}N^{-1}$, where without loss of generality we can assume that $M,N\in\M_n(\cO_L)$, where
$\cO_L$ denotes the ring of integers of $L$. 
Let $\la_1,\ldots,\la_s\in\cO_L$ (resp., $\mu_1,\ldots,\mu_t\in\cO_L$) be all the distinct eigenvalues of $A$ (resp., $B$), $1\leq s\leq n$, $1\leq t\leq n$.  
Let 
$A=\sum_{i=1}^sA_i$ (resp., $B=\sum_{i=1}^tB_i$), where $A_i,B_i\in\M_n(L)$, 
$$
A_i=A(J_i)=MD_{A,i}M^{-1},\quad 
D_{A,i}=D(J_i)=\left(
\begin{array}{ccccccc}
\ddots &        &        &        &        &        &        \\
       & \bm{0} &        &        &        &        &        \\
       &        & J_i    &        &        &        &        \\
       &        &        & \bm{0} &        &        &        \\
       &        &        &        & \ddots &        &        \\
\end{array}
\right),
$$
$$
J_i =
\begin{bmatrix}
\lambda_i & 1       & 0       & \cdots & 0 \\
0       & \lambda_i & 1       & \cdots & 0 \\
\vdots  & \ddots  & \ddots  & \ddots & \vdots \\
0       & \cdots  & 0       & \lambda_i & 1 \\
0       & \cdots  & \cdots  & 0       & \lambda_i
\end{bmatrix},
$$
$\bm{0}$ represents zero matrices, and similarly for $B_j$, $1\leq j\leq t$. 
For a prime ideal $\fp$ of $\cO_L$, let $\cO_{\fp}$ denote the localization of $\cO_L$ at $\fp$. 
For $j\in\bZ$, 
note that $A^j=\sum_{i=1}^sA^j_i$, where we denote $A_i^{-1}=A(J_i^{-1})$. For $T\in\M_n(\bZ)$, from the definition of $G_A$, 
we have that   
$T(G_A)\subseteq G_B$ if and only if for any $k\in\bZ$ there exists
$l_k\in\bN\cup\{0\}$ such that $B^{l_k}TA^k\in\M_n(\bZ)$ if and only if $B^{l_k}TA^k\in\M_n(\bZ_{(p)})$ for any $p\in\cP(A)$, since there are only primes dividing $\det A$ in the denominators of $B^{l_k}TA^k$ (recall that $T\in\M_n(\bZ)$ and $l_k\geq 0$). 
For each $p\in\cP(A)$, we choose one prime ideal $\fp$ of $\cO_L$ lying above $p$. Let $\cS$ denote the set of all those prime ideals of $\cO_L$. Then $T(G_A)\subseteq G_B$ if and only if for any 
$\fp\in\cS$ and any $k\in\bZ$ there exists $l_k\in\bN\cup\{0\}$ such that $B^{l_k}TA^k\in\M_n(\cO_{\fp})$. 
For $k\in\bZ$, we write 
\begin{eqnarray*}
(A_{\fp})^k&=&\sum_{i:\,\fp\vert\la_i}A_i^k, \quad (A_{\fp}')^k=\sum_{j:\,\fp\nmid\la_j}A_j^k, \\
A^k&=&(A_{\fp})^k+(A_{\fp}')^k, 
\end{eqnarray*}
and similarly for $B$. 
Let $\cX_{A,\fp}$ (resp., $\cX_{B,\fp}$) denote the span over $L$ of the columns of $M$ (resp., $N$) corresponding to eigenvalues of $A$ (resp., $B$) divisible by $\fp\in\cS$ in $\cO_L$. It follows from the proof of \cite[Theorem 4.3]{s2} that if $T(G_A)\subseteq G_B$, then $T(\cX_{A,\fp})\subseteq \cX_{B,\fp}$ for each $\fp\in\cS$, which are computable conditions. 
Thus, we now assume that 
$T(\cX_{A,\fp})\subseteq \cX_{B,\fp}$ for each $\fp\in\cS$, since otherwise $T(G_A)$ is not inside $G_B$. 
Equivalently, 
$(B_{\fp}')^{l_k}T(A_{\fp})^k={\bm{0}}$, the zero matrix. Also, note that for any $k\in\bZ$ there exists $l_k\in\bN$ such that 
$(B_{\fp})^{l_k}TA^k\in\M_n(\cO_{\fp})$, since for $d$ big enough and $\la_i$ divisible by $\fp$, we have that $B_i^d\in\M_n(\fp)$.  
Thus,
$$
B^{l_k}TA^k=\left((B_{\fp})^{l_k}+(B_{\fp}')^{l_k}\right)T\left((A_{\fp})^{k}+(A_{\fp}')^{k}\right)=
(B_{\fp})^{l_k}TA^k+(B_{\fp}')^{l_k}T(A_{\fp}')^{k},
$$
and for any $k\in\bZ$ there exists $l_k\in\bN$ such that 
$B^{l_k}TA^k\in\M_n(\cO_{\fp})$ if and only if for any $k\in\bZ$ there exists $l_k\in\bN$ such that
\bbe\label{eq:last0}
(B_{\fp}')^{l_k}T(A_{\fp}')^{k}\in\M_n(\cO_{\fp}).
\ee 
It might happen that $\fp$ divides the determinant of $M$ and hence $A_i$ might not be integral over $L$. Similarly, with $A_i^k$ for an integer $k$, even when $\fp$ does not divide $\la_i$. 
However, the valuations of the entries of $A_i^k$ at $\fp$ are bounded below by the valuation of $\det M$ at $\fp$ for $i$ such that $\la_i$ is not divisible by $\fp$. Similarly, for $B$. 
More precisely, let $r=\val_{\fp}(\det M)+\val_{\fp}(\det N)$. 
Since
for each $\la_i$ (resp., $\mu_j$) not divisible by $\fp$, the image of $J_i$ (resp., $\tilde{J}_j$, the Jordan cell of $B$ corresponding to $\mu_j$) in $\cO_L/\fp^r$ is non-singular, we have that the image of 
$\{J^k_i \}_{k\in\bZ}$ (resp., $\{\tilde{J}^k_j \}_{k\in\bN}$) in $\cO_L/\fp^r$ is a finite group. Thus, in \eqref{eq:last0}, one checks only finitely many pairs $(l_k,k)$. This shows that there are finitely many computable steps in deciding whether $T(G_A)\subseteq G_B$.  Hence,  elements of $\End(G_A)$, $\Aut(G_A)$, and 
$\Hom(G_A,G_B)$ are computable. 
\end{proof}

\begin{example}
Let $A\in\M_3(\bZ)$ be a matrix with characteristic polynomial $h_A=(x-2)(x-3)^2$. It is known that there exists 
$S\in\GL_3(\bZ)$ such that $SAS^{-1}$ is upper-triangular. Thus, without loss of generality, we can assume that $A$ is upper-triangular. For example,
$$
A=
  \begin{pmatrix}
    2 & 1 & 3 \\
    0 & 3 & 1 \\
    0 & 0 & 3
  \end{pmatrix},\quad A=MJM^{-1},\quad
  M=
  \begin{pmatrix}
    2 & 1 & 2 \\
    0 & 1 & 0 \\
    0 & 0 & 1
  \end{pmatrix},\quad
J=
  \begin{pmatrix}
    2 & 0 & 0 \\
    0 & 3 & 1 \\
    0 & 0 & 3
  \end{pmatrix}.  
  $$
  Then $\End(G_A)$ is generated as a left and right $\bZ[A]$-module by 
  $$
X=X(a,Y)=
  \begin{pmatrix}
    a & 0 & 0 \\
    0 & d & e \\
    0 & f & g
  \end{pmatrix},\quad a,d,e,f,g\in\bZ,\quad 
  Y=
  \begin{pmatrix}
d & e \\
f & g
  \end{pmatrix}.  
  $$
 Hence, $\End(G_A)=\{MXM^{-1}\}$, where $X=X(a,Y)$ with  $a\in\bZ[\frac{1}{2}]$ and $Y\in\M_2(\bZ[\frac{1}{3}])$, and
$\Aut(G_A)=\{MXM^{-1}\}$, where $X=X(a,Y)$ with  $a\in\cO_{\bQ,2}^{\times}$ and 
$Y\in\GL_2(\bZ[\frac{1}{3}])$.
\end{example}


By Lemma \ref{l:easy0}, we may assume for the remainder of the section that 
$A\not\in\GL_n(\bZ)$ and $\cP'(A)\ne\emptyset$ in order to compute $\End(G_A)$ and $\Aut(G_A)$. For a prime $p\in\bN$, we denote by 
$t_p=t_p(A)$ the multiplicity of zero in the reduction of characteristic polynomial $h_A$ of $A$ modulo $p$, $0\leq t_p\leq n$. In particular, $p\in\cP$ if and only if $t_p\ne 0$ and $p\in\cP'$  if and only if  $0<t_p< n$.

\subsection{Irreducible characteristic polynomial} 
 In this work, we concentrate on the most explicit case, when 
$n\geq 2$ is arbitrary, the characteristic polynmial $h_A$ of $A$ is irreducible, and there exists a prime $p\in\bN$
such that $n$, $t_p$ are coprime, denoted by 
$(n,t_p)=1$.  It is known from our previous work \cite{s2} that under these assumptions, the structure of $G_A$ is especially rigid. 
That allows us to make calculations especially concrete. 

\begin{thm}[{\it c.f.}, \cite{kr}]\label{th:endcomp}
Let $A\in\M_n(\bZ)$ be non-singular. Assume that characteristic polynomial $h_A\in\bZ[t]$ of $A$ is irreducible and there exists a prime $p\in\bN$ such that $(n,t_p)=1$. 
Then $\End(G_A)$, $\Aut(G_A)$ are computable.
\end{thm}
\begin{proof}
Let $A\in\M_n(\bZ)$ be given. 
We first find the characteristic polynomial of $A$ and check whether it is irreducible, which is known to be a computable condition (see {\it e.g.}, \cite[Algorithm 16.22]{gg}). We then compute $t_p$ for all $p\in\cP$ and verify the second condition 
$(n,t_p)=1$. All of these steps are computable. 
Let $\overline{\bQ}$ denote a fixed algebraic closure of $\bQ$ and let $\la\in\overline{\bQ}$ be an eigenvalue of $A$, $K=\bQ(\la)$, and let ${\bf u}\in K^n$ be an eigenvector of $A$ corresponding to $\la$.  
Multiplying ${\bf u}$ by an appropriate integer, without loss of generality, we can assume 
${\bf u}\in\bZ[\la]^n\subseteq (\cO_K)^n$, where $\cO_K$ is the ring of integers of $K$.  
Since $h_A$ is irreducible,  the absolute Galois group 
$\Gal(\overline{\bQ}/\bQ)$ acts transitively on the set of all eigenvalues of $A$, {\it i.e.}, there exist
$\sg_1,\ldots,\sg_n\in \Gal(\overline{\bQ}/\bQ)$, $\sg_1=\id$, such that $A=M\La M^{-1}$, where
\bbe\label{eq:lm0}
\La=\diag\left(\begin{matrix} \sg_1(\la)& \ldots & \sg_n(\la)\end{matrix}\right),\quad M=\left(\begin{matrix} \sg_1({\bf u}) & \ldots & \sg_n({\bf u}\end{matrix})\right)
\ee
with each $\sg_i({\bf u})$ written as a column, $i\in\{1,\ldots,n\}$.  By \cite[Proposition 4.1]{s3}, for any 
$T\in\End(G_A)$ there exists $x\in\cO_{K,\la}$ such that  
\bbe\label{eq:tx0}
T=T(x)=MXM^{-1},\quad X=\diag\left(\begin{matrix} \sg_1(x)&\ldots &\sg_n(x)\end{matrix}\right).
\ee 
Here, $x\in\cO_{K,\la}$ if and only if $x\in K$ and there exists $\la^i$,
$i\in\bN\cup\{0\}$, such that $\la^ix\in\cO_K$. Note that $\bZ[A,A^{-1}]\subseteq \End(G_A)$ and hence $\End(G_A)$ is a $\bZ[A,A^{-1}]$-module. 
Note that $\bZ[\la^{\pm 1}]=\bZ[\la^{-1}]$, since $\la$ is an algebraic integer. 
Thus, $\End(G_A)$ is a $\bZ[\la^{-1}]$-module, where for $T\in\End(G_A)$, the action of $\la$ (respectively, of $A$) is
given by $\la\cdot T(x)=T(\la x)$ (respectively, $A\cdot T=AT$). Note that 
$A$ and $T$ commute.
By \cite[Corollary 4.3]{s3}, $\End(G_A)$ is a finitely-generated $\bZ[\la^{-1}]$-module and we will describe how to compute its finitely many generators as such. Let $T\in\End(G_A)$ be arbitrary, so that $T=T(x)$ by \eqref{eq:tx0}. Since we are looking for generators of $\End(G_A)$ as a 
$\bZ[\la^{-1}]$-module, without loss of generality, we can assume that $x$ itself is an algebraic integer, {\it i.e.}, $x\in\cO_K$. 
Moreover, since $T(G_A)\subseteq G_A$, there exists $k\in\bN$ such that $A^kT\in\M_n(\bZ)$. Thus, without loss of generality, we can assume that $T\in\M_n(\bZ)$. Thus, it is enough to find all $x\in\cO_K$ such that $T(x)\in\M_n(\bZ)$. 
Clearly, if $x\in\bZ[\la]$, then $T(x)\in\bZ[A]\subset \M_n(\bZ)$. 
Therefore, one only needs to consider a finite set 
$\cG\subset\cO_K$ of representatives of the quotient $\cO_K/\bZ[\la]$, {\it i.e.}, whether 
$T(x)\in \M_n(\bZ)$ for each 
 $x\in\cG$. Thus, all generators of $\End(G_A)$ as a $\bZ[\la^{-1}]$-module can be found in finitely many steps, provided that each step in the proof is computable. 
Indeed, algorithms exist to compute an integral basis of $\cO_K$, starting from $1,\la,\ldots,\la^{n-1}$,  
as well as the index $l=[\cO_K:\bZ[\la]]$ of $\bZ[\la]$ in $\cO_K$ and set $\cG$ ({\it e.g.}, Algorithm $A_6$ \cite{gr}). Thus, each $x\in\cG$ can be written as $x=\sum_{i=0}^{n-1}x_{i}\la^i$, $x_{0},\ldots,x_{n-1}\in\bQ$, and 
$T(x)=\sum_{i=0}^{n-1}x_{i}A^i$. 
This completes a proof that $\End(G_A)$ is computable. 

\sbr

We now make a few observations that will be used later in the proof. 
For $x\in\cO_K$ and $l=[\cO_K:\bZ[\la]]$, by \eqref{eq:lm0} and \eqref{eq:tx0}, we have that $l\cdot T(x)\in\bZ[A]$.  
Furthermore, the discriminant of the lattice in $\cO_K$ generated by coordinates of ${\bf u}\in (\cO_K)^n$ gives $m=(\det M)^2\in\bZ$. 
Thus, $m\cdot T(x)\in\M_n(\bZ)$. Note that $A$ and $T$ commute, so that $T(G_A)\subseteq G_A$ if and only if there exists $j\in\bN\cup\{0\}$ such that $A^jT\in\M_n(\bZ)$ if and only if for each prime $p\in\bN$ there exists $j=j(p)\in\bN\cup\{0\}$ such that $A^jT\in\M_n(\bZ_{(p)})$, where 
$$
\bZ_{(p)}=\left.\left\{\frac{q}{r}\in\bQ \,\, \right\vert \, q,r\in\bZ,\,\,r\ne 0,\,\,(p,r)=1\right\},
$$
a subring of $\bQ$. Putting it all together, for $T=T(x)$ with $x\in\cO_K$, we have that $T\in\End(G_A)$ if and only if for each prime 
$p\in\bN$ dividing both $l$ and $m$ ({\it i.e.}, $p$ divides the greatest common divisor $(l,m)$ of $l$ and $m$) there 
exists $j=j(p)\in\bN\cup\{0\}$ such that $A^jT\in\M_n(\bZ_{(p)})$. 
It is enough to check $j=j(p)$ bounded by the order of the image of the set 
$\{\la^i\,\vert\,i\in\bN\cup\{0\}\}$ in the finite ring $\cO_K/(p^{g_p}\cO_K)$, $g_p={\min\{\val_pl,\val_pm\}}$. In particular, if
$l$ and $m$ are coprime, then $\End(G_A)\cong\cO_{K,\la}$ via $T(x)\mapsto x$.

\sbr

We now show computability of $\Aut(G_A)$. By \cite[Corollary 4.2]{s3}, $\Aut(G_A)$ is a finitely-generated  abelian group. 
Denote 
$$
\cO^{\times}_{K,\la}=\{x\in K-\{0\}\,\vert\, \val_{\fp}(x)=0\text{ for any prime ideal }\fp\text{ of }\cO_K\text{ not dividing }\la\}.
$$
In other words, $x\in K-\{0\}$ belongs to $\cO^{\times}_{K,\la}$ if and only if there are only prime ideals dividing $\la$
in the prime decomposition of ideal $(x)=x\cO_K$. By \cite[Proposition 4.1]{s3}, if $T(x)\in\Aut(G_A)$, then $x\in\cO^{\times}_{K,\la}$. The converse does not hold in general. The group $\cO^{\times}_{K,\la}$ is an example of a group of $S$-units in a number field $K$, where $S$ is a finite set of prime ideals of $\cO_K$. It is well-known that those groups are finitely generated and there are algorithms to determine generating sets. The subgroup 
$$
Y=\{x\in \cO^{\times}_{K,\la}\,\vert\,T(x)\in\Aut(G_A)\}\cong\Aut(G_A)
$$ 
is finitely generated, too. Assuming the group of units $\cO_K^{\times}$ of ring $\cO_K$ is known, here is one way to find generators of $\cO^{\times}_{K,\la}$. 
Clearly, $\la\in Y$, since 
$T(\la)=A$. Let 
$T(x)\in\Aut(G_A)$, $x\in\cO_{K,\la}^{\times}$, so that there exists a power of $\la$ such that 
$\la^jx\in\cO_K$, $j\in\bN\cup\{0\}$. Without loss of generality, we can assume that $x\in\cO_K\cap\cO^{\times}_{K,\la}$. 
Then there is 
$i\in\bN\cup\{0\}$ and $y\in\cO_K-\{0\}$ such that $\frac{\la^i}{x}=y$, equivalently, $xy=\la^i$. We now decompose 
the ideal of $\cO_K$ generated by $\la$ into (distinct) prime ideals: $(\la)=\la\cO_K=\fp_1^{i_1}\cdots\fp_k^{i_k}$, where
$i_1,\ldots,i_k\in\bN$, $\fp_1,\ldots,\fp_k$ are distinct prime ideals of $\cO_K$. Let $l_i=\abs{\fp_i}$, $l_i\in\bN$, be the order of $\fp_i$ in the class group of $K$, {\it i.e.}, $\fp_i^{l_i}=(z_i)$ is principal, $z_i\in\cO_K$, $i\in\{1,\ldots,k\}$.
From the uniqueness of prime decompositions and $xy=\la^i$, we have that $(x)=\fp_1^{j_1}\cdots\fp_k^{j_k}$,
$j_1,\ldots,j_k\in\bN\cup\{0\}$. Thus, the product $\fp_1^{j_1}\cdots\fp_k^{j_k}$ is principal. 
It is enough to find all principal products 
of this form among finitely many choices when 
$0\leq j_i<l_i$, $i\in\{1,\ldots,k\}$. For each such principal product we fix  a generator. We denote the (finite) set of such generators by $\cG\subset\cO_K$. Thus, 
$x$ can be written as a product 
\bbe\label{eq:xx0}
x=ugz_1^{a_1}\cdots z_k^{a_k},\quad a_1,\ldots,a_k\in\bN\cup\{0\},\quad u\in\cO_K^{\times},\quad g\in\cG.
\ee
It is well-known that $\cO_K^{\times}$ is finitely generated and there are algorithms to find its finite generating sets (see {\it e.g.}, \cite[Algorithm $A_8$]{gr}). 

\sbr

We now assume that a finite generating set $\{f_1,\ldots,f_r\}$ of $\cO_{K,\la}^{\times}$ is known. 
For each
$f_i$, there exist $s_i,t_i\in\bN\cup\{0\}$ such that $\la^{s_i}f_i,\la^{t_i}f_i^{-1}\in\cO_K$, $i\in\{1,\ldots,r\}$. 
For any finite subset $I\subseteq\{1,\ldots,r\}$ and its complement $I^c$ in $\{1,\ldots,r\}$, let 
$$
x=\prod_{i\in I}f_i^{a_i}\times\prod_{i\in I^c}(f_i^{-1})^{a_i},\quad a_1,\ldots,a_r\in\bN\cup\{0\},
$$
an arbitrary element of $\cO_{K,\la}^{\times}$. 
Then 
\begin{eqnarray*}
 \la^{\al}x & = & \prod_{i\in I}(\la^{s_i}f_i)^{a_i}\times\prod_{i\in I^c}(\la^{t_i}f_i^{-1})^{a_i}\in\cO_K, \quad \al=\sum_{i\in I} s_ia_i+\sum_{i\in I^c} t_ia_i,\\
 \la^{\be}x^{-1} & = & \prod_{i\in I}(\la^{t_i}f_i^{-1})^{a_i}\times\prod_{i\in I^c}(\la^{s_i}f_i)^{a_i}\in\cO_K, \quad \be=\sum_{i\in I} t_ia_i+\sum_{i\in I^c} s_ia_i.
\end{eqnarray*}
Denote $N=\prod_{p\,\vert(l,m)}p^{g_p}$, where $g_p={\min\{\val_pl,\val_pm\}}$ as above. It suffices to check whether $T(\la^{\al}x)\in\End(G_A)$ (equivalently, 
$T(x)\in\End(G_A)$) for finitely many values of $\{a_1,\ldots,a_r\}$. Indeed, it is enough to consider each
$a_i$, $i\in I$ (resp., $i\in I^c$), bounded by the order of the image of the set $\{(\la^{s_i}f_i)^{b}\}_{b\in\bN\cup\{0\}}$ (resp., $\{(\la^{t_i}f_i^{-1})^{b}\}$)  in the finite ring $\cO_K/N\cO_K$. (As mentioned above, we do not address the effectiveness of the algorithm in this paper. However, to simplify computations -- similarly to the case of $\End(G_A)$ discussed earlier -- one could bound the degrees $a_i$'s by the corresponding orders in the quotient rings $\cO_K/(p^{g_p}\cO_K)$, where $p$ runs over the primes dividing $(l,m)$.) The results are 
congruences for powers $a_1,\ldots,a_r$. We repeat with $x^{-1}$, {\it i.e.}, whether 
$T(\la^{\be}x^{-1})\in\End(G_A)$ (equivalently, 
$T(x^{-1})\in\End(G_A)$) for finitely many values of $\{a_1,\ldots,a_r\}$, where each 
$a_i$, $i\in I$ (resp., $i\in I^c$), is bounded by the order of the image of the set 
$\{(\la^{t_i}f_i^{-1})^{b}\}_{b\in\bN\cup\{0\}}$ (resp., $\{(\la^{s_i}f_i)^{b}\}$)  in $\cO_K/N\cO_K$. Note that if some $f_i$ is a unit in $\cO_K$, {\it i.e.}, 
$f_i\in\cO_K^{\times}$, then $s_i=t_i=0$ and the image of $\{1,f_i^{\pm 1},f_i^{\pm 2},\ldots\}$ in $\cO_K/N\cO_K$ is a (finite) group.
Thus, to check whether $ T(\la^{\al}x), T(\la^{\be}x)\in\End(G_A)$, it is enough to assume that 
$a_i$ is bounded by the order of $f_i$ in $\cO_K/N\cO_K$.
The final step involves solving, for each power $a_i$, $i\in\{1,\ldots,r\}$, a system of at most two congruences when working over the ring $\cO_K/N\cO_K$, and potentially more when considering the rings $\cO_K/(p^{g_p}\cO_K)$. Solving a system of congruences is computable, 
either by applying the Chinese Remainder Theorem when the moduli are coprime, or by using more general methods such as the extended Euclidean algorithm and verifying compatibility when the moduli are not coprime.

\sbr

We used the following steps above, all of which are computable. In particular, decomposing the ideal $(\la)$ into prime ideals (see {\it e.g.}, \cite[Algorithm $A_{10}$]{gr}), deciding if a product of prime ideals is principal and if yes, producing a generator (see {\it e.g.}, \cite[Algorithm 6.5.10]{cohen})), and computing the class number of $K$ (see {\it e.g.}, \cite[Algorithm $A_{12}$]{gr}). 
\end{proof}

\begin{rem}\label{rem:centr} 
It follows from the above that under the conditions that the characteristic polynomial of $A$ is irreducible and there exists a prime $p\in\bN$ such that $(n,t_p)=1$, as a $\bZ[A,A^{-1}]$-module, $\End(G_A)$ is generated by elements of the centralizer of $A$ in $\M_n(\bZ)$, which gives another way to compute $\End(G_A)$ as a $\bZ[A,A^{-1}]$-module. However, it is not clear how to describe $\Aut(G_A)$ (equivalently, 
the unit group of  $\End(G_A)$), using this approach.   
\end{rem}

The algorithms mentioned in the proof of Theorem \ref{th:endcomp} are implemented in \cite{sage}, and we use it together with \cite{db} to produce examples. We demonstrate with two examples ($3$-dimensional and $5$-dimensional). One can find $2$-dimensional examples in \cite{s3}.


\begin{example}\label{ex:30}
Let 
$$
A=\left(\begin{matrix}
0 & 1 & 0 \\
0 & 0 & 1 \\
-59 & 15 & 1 \\
\end{matrix}
\right),\quad  
B=\left(\begin{matrix}
0 & 1 & 0 \\
7 & -4 & 6 \\
-4 & -2 & 5 \\
\end{matrix}
\right).
$$
Both matrices have characteristic polynomial $h=t^3-t^2-15t+59$, irreducible in $\bZ[t]$, $59$ is a prime, and $t_{59}=1$, so that the hypotheses of Theorem 
\ref{th:endcomp} hold. 
For a root $\la$ of $h$, 
it is known that $\cO_K=\bZ[1,\la,\frac{1}{2}\la^2-\frac{1}{2}]$ and $[\cO_K:\bZ[\la]]=2$ \cite{db}. Thus, we can take 
$x=\frac{1}{2}\la^2-\frac{1}{2}$ as a non-trivial representative for $\cO_K/\bZ[\la]$. 
One calculates that $\frac{1}{2}A^2-\frac{1}{2}A^0$ is not integer, but 
$\frac{1}{2}B^2-\frac{1}{2}B^0$ is integer. Thus,
$\End(G_A)=\bZ[A^{-1}]$ and $\End(G_B)\cong\cO_{K,\la}$ via $T(x)\in\End(G_B)\mapsto x\in\cO_{K,\la}$.
Explicitly, 
$\End(G_B)$ is generated by $T_0=\frac{1}{2}B^2-\frac{1}{2}B^0$ as a $\bZ[\la^{-1}]$-module, {\it i.e.},
$$
\End(G_B)=\left\{\sum_i b_iB^{m_i}+\sum_j c_jB^{n_j}T_0\,\,\Big\vert \,\, \forall \, b_i,m_i,c_j,n_j\in\bZ,\, i,j\in\bN  \right\},
$$
where each sum has finitely many non-zero terms. Now, one can compute $\Aut(G_A)$, $\Aut(G_B)$ as unit groups of 
$\End(G_A)$, $\End(G_B)$, respectively. For example, immediately, $\Aut(G_B)\cong\cO_{K,\la}^{\times}$. In general, it is not clear whether the computability of $\End(G_A)$ implies the computability of $\Aut(G_A)$. Thus, in Example \ref{ex:3'} below,
we compute $\Aut(G_A)$ 
according to the procedure described in Theorem 
\ref{th:endcomp}. 
\end{example}

\begin{example}\label{ex:50}
Let $h=t^5-2t^4-t^3+5t^2+9$, irreducible in $\bZ[t]$, $t_{3}=2$, and $\caR=\bZ[3^{-1}]$. 
It is known that $\cO_K=\bZ[1,\la,\la^2,\la^3,g]$, $g=g(\la)=\frac{1}{6}\la^4+\frac{1}{6}\la^3+\frac{1}{3}\la^2-\frac{1}{6}\la-\frac{1}{2}$, and 
we can take 
$g$ as a generator of $\cO_K/\bZ[\la]$, $[\cO_K:\bZ[\la]]=6$, $\la$ is a root of $h$ \cite{db}.  
Let
$$
A=\left(\begin{matrix}
0 & 1 & 0 & 0 & 0 \\
0 & 0 & 1 & 0 & 0 \\
0 & 0 & 0 & 1 & 0 \\
0 & 0 & 0 & 0 & 1 \\
-9 & 0 & -5 & 1 & 2
\end{matrix}
\right),\quad 
B=\left(\begin{matrix}
-1 & 2 & 0 &  0 & 0 \\
-1 &  2 & 0 & 0 & 1 \\
 1 & 1 & 0 & 0 & 0 \\
-2 & 1 & -1 & 1 & 0 \\
-4 & 2 & 1 & 2 & 0 
\end{matrix}\right).
$$
Both matrices have characteristic polynomial $h$, so that the hypotheses of Theorem \ref{th:endcomp} hold. As in Example 
\ref{ex:30} above, 
one can check that $iT(g)=i(\frac{1}{6}A^4+\frac{1}{6}A^3+\frac{1}{3}A^2-\frac{1}{6}A-\frac{1}{2}A^0)$ is not integer for any 
$i\in\{1,\ldots,5\}$. 
Thus, 
$\End(G_A)=\bZ[A^{-1}]$.  
On the other hand, $g(B)$ has integer entries, hence 
$\End(G_B)\cong\cO_{K,\la}$  via $T(x)\in\End(G_B)\mapsto x\in\cO_{K,\la}$. 
Note that, immediately, $\Aut(G_B)\cong\cO_{K,\la}^{\times}$ as the unit group of $\cO_{K,\la}$. 
In Example \ref{ex:5'} below, we compute $\Aut(G_A)$ 
independently of
the corresponding endomorphism ring according to the procedure described in Theorem 
\ref{th:endcomp}. 
%
\end{example}

\begin{rem}
In the second example, 
we found matrices $A$, $B$ from ideals generating the class group of $\cO_K$.
It is known that $\GL_n(\bZ)$-conjugacy classes of matrices with characteristic polynomial $h$ are in one-to-one correspondence with $\bZ[\la]$-ideal classes of fractional ideals of $K$, which is a finite set. 
Clearly, an $\cO_K$-ideal is a $\bZ[\la]$-ideal, so that the number of $\bZ[\la]$-ideal classes is bigger than the number of 
$\cO_K$-ideal classes. 
If one knows generators of
all $\bZ[\la]$-ideal classes of $K$,  
then one can describe $\End(G_A)$ as a $\bZ[\la^{-1}]$-module for any $A\in\M_n(\bZ)$ with characteristic polynomial $h$. 
\end{rem}

\subsection{$\bZ^n$-odometers} 
Theorem \ref{th:endcomp} has applications to $\bZ^n$-odometers. Namely, it provides a way to compute its endomorphism rings 
under the assumptions of the theorem. In \cite[Remark 7.6]{s3}, we already show that the endomorphism rings are computable, 
using the fact that the centralizer of an integer matrix is computable. We now present a more direct approach to examining its computability.

\sbr

Recall that $\bZ^n$-odometer is a dynamical system 
consisting of a topological space $X$ and an action of the group $\bZ^n$ on $X$ (by homeomorphisms). 
 For the groups (under addition)
$$
G_i=\left.\left\{A^{i}{\bf x}\,\right\vert\, {\bf x}\in\bZ^n\right\},\quad i\in\bN,
$$
consider a decreasing sequence 
$$
G=\bZ^n\supseteq G_1\supseteq G_2\supseteq\cdots
$$
and the natural maps $\pi_i:G/G_{i+1}\rar G/G_{i}$, $i\in\bN$. The associated $\bZ^n$-odometer is  the inverse limit 
\bbe\label{eq:xg}
X_A=\lim_{\longleftarrow} \left(G/G_{i}\right)
\ee
together with the natural action of $\bZ^n$. 
In \cite{cp}, the authors study the linear representation group of 
$X_A$ denoted by $\vec{N}(X_A)$. By \cite[Lemma 2.6]{cp}, $\vec{N}(X_A)$ consists of $T\in\GL_n(\bZ)$ ({\it i.e.}, $T\in\M_n(\bZ)$ with $\det T=\pm 1$) such that for any $m\in\bN\cup\{0\}$ there exists $k_m\in\bN\cup\{0\}$ with
\bbe\label{eq:linrepgr}
A^{-m}TA^{k_m}\in\M_n(\bZ).
\ee
By taking the transpose of the condition, one can see that it is
 equivalent to the condition that $T^t$ defines an endomorphism of $G_{A^t}$. 
\begin{lem}[Lemma 7.1 \cite{s3}]\label{lem:nvec}
$T\in\vec{N}(X_A)$ if and only if $T^t\in\End(G_{A^t})\cap\GL_n(\bZ)$.
\end{lem} 

We now apply our results on computability of $\End(G_{A^t})$.

\begin{thm}\label{th:endcompodom}
Let $A\in\M_n(\bZ)$ be non-singular. Then elements of $\vec{N}(X_A)$ are computable. 
Assume that characteristic polynomial $h_A\in\bZ[t]$ of $A$ is irreducible and there exists $p\in\bN$ such that $(n,t_p)=1$. 
Then $\vec{N}(X_A)$ is computable.
\end{thm}
\begin{proof}
The computability of elements of $\vec{N}(X_A)$ follows from Proposition \ref{pr:decide} together with Lemma \ref{lem:nvec}.
The computability of $\vec{N}(X_A)$ follows from Theorem \ref{th:endcomp} together with Lemma \ref{lem:nvec}. More precisely, 
in the notation of the proof of 
Theorem \ref{th:endcomp}, 
by \cite[Proposition 7.4 (2)]{s3}, for any $T\in\vec{N}(X_A)$ there exists $x\in\cO_{K}^{\times}$ such that 
$T(x)=MXM^{-1}$. Conversely, if $T=T(x)$ for some $x\in\cO_{K}^{\times}$ and $T\in\M_n(\bZ)$, then $T\in \vec{N}(X_A)$.
(Note that it is always the case that given $x\in\cO_{K}^{\times}$, $T=T(x)\in\GL_n(\bQ)$ and $\det T=\pm 1$.) 
%
%
\end{proof}

We show the process on two previous Examples  \ref{ex:30} and \ref{ex:50} above, and we also compute $\Aut(G_A)$.
\begin{example}\label{ex:3'}
As in Example  \ref{ex:30} above,  
let $h=t^3-t^2-15t+59$, irreducible in $\bZ[t]$, $59$ is a prime, $t_{59}=1$. 
It is known that $\cO_K=\bZ[1,\la,\frac{1}{2}\la^2-\frac{1}{2}]$, 
and the unit group is generated by $-1$ and a (fundamental) unit $f\in\bZ[\la]$, so that $\cO_K^{\times}\subset\bZ[\la]$  \cite{db}. Thus, $\vec{N}(X_A)\cong\cO_K^{\times}$ via $T(x)\in\vec{N}(X_A)\mapsto x\in\cO_K^{\times}$ for any $A\in\M_3(\bZ)$ with characteristic polynomial $h$. Also, $(\la)$ is prime. Thus, in the notation of Theorem \ref{th:endcomp}, $x=u\la^i$, $y=u^{-1}\la^j$, where $u\in\cO_K^{\times}$, $i,j\in\bN\cup\{0\}$. 
Thus, $x,y\in\bZ[\la]$ and hence 
$$
\Aut(G_A)=\Aut(G_B)\cong\cO_{K,\la}^{\times},\quad \cO_{K,\la}^{\times}=\cO_K^{\times}\times\{\la^i\,\vert\,i\in\bZ\}\cong(\bZ/2\bZ)\times\bZ^2,
$$
 via $T(x)\in\Aut(G_A)\mapsto x\in\cO_{K,\la}^{\times}$.
\end{example}

\begin{example}\label{ex:5'}
As in Example \ref{ex:50} above, 
let $h=t^5-2t^4-t^3+5t^2+9$, irreducible in $\bZ[t]$ and $t_{3}=2$. 
It is known that 
$$
\cO_K=\bZ\left[1,\la,\la^2,\la^3,\frac{1}{6}\la^4+\frac{1}{6}\la^3+\frac{1}{3}\la^2-\frac{1}{6}\la-\frac{1}{2}\right].
$$ 
Let
$$
A=\left(\begin{matrix}
0 & 1 & 0 & 0 & 0 \\
0 & 0 & 1 & 0 & 0 \\
0 & 0 & 0 & 1 & 0 \\
0 & 0 & 0 & 0 & 1 \\
-9 & 0 & -5 & 1 & 2
\end{matrix}
\right),\quad 
B=\left(\begin{matrix}
-1 & 2 & 0 &  0 & 0 \\
-1 &  2 & 0 & 0 & 1 \\
 1 & 1 & 0 & 0 & 0 \\
-2 & 1 & -1 & 1 & 0 \\
-4 & 2 & 1 & 2 & 0 
\end{matrix}\right).
$$
Both matrices have characteristic polynomial $h$. First, we compute $\Aut(G_A)$. 
Here $(\la)=\fp_1\fp_2$.
The group $\cO_{K,\la}^{\times}$ is generated by five generators 
\begin{eqnarray*}
&& -1,\,\, \la,\\
&& f_1=f_1(\la)=\frac{2}{3}\la^4 - \frac{1}{3}\la^3 - \frac{11}{3}\la^2 + \frac{10}{3}\la + 9\in\cO_K^{\times}, \\
 && f_2=f_2(\la)=2\la^4 - 4\la^3 - 4\la^2 + 14\la + 5\in\cO_K^{\times}, \\
&& f_3=f_3(\la)=\frac{5}{6}\la^4 + \frac{29}{6}\la^3 + \frac{23}{3}\la^2 + \frac{43}{6}\la + \frac{13}{2}\in\cO_K,
\end{eqnarray*}
and 
$\la^3f_3^{-1}\in\cO_K$ \cite{sage}. In the notation of the proof of Theorem \ref{th:endcomp}, $s_1=t_1=s_2=t_2=s_3=0$, $t_3=3$. 
Denote $T_i=f_i(A)$, $i\in\{1,2,3\}$. 
For $A$, ${\bf u}=\left(\begin{matrix} 1& \la& \ldots & \la^4\end{matrix}\right)^t\in(\cO_K)^5$ is an eigenvector of $A$ corresponding to $\la$, the discriminant of the lattice in $\cO_K$ generated by coordinates of ${\bf u}$ gives 
$m=(\det M)^2$. We compute $m=(\det M)^2=2^4\cdot 3^3\cdot 130643$, $l=[\cO_K:\bZ[\la]]=2\cdot 3$, where $130643$ is prime. 
Then $(l,m)=6$, $g_2=g_3=1$, $N=6$. The orders of the images of the sets 
$\{(f_3)^i \}_{i\in\bN\cup\{0\}}$, $\{(\la^3f_3^{-1})^i\}_{i\in\bN\cup\{0\}}$ in $\cO_K/N\cO_K$ are both $13$, 
the orders of $f_1,f_2$ in $\cO_K/N\cO_K$ are both $6$, the order of the image of the set $\{\la^i\}_{i\in\bN\cup\{0\}}$ in $\cO_K/N\cO_K$ is $7$. Then for an arbitrary element $x\in\cO_{K,\la}^{\times}$, we have that 
$$
x=\al f_1^{a_1}f_2^{a_2}f_3^{a_3}\la^{a_4},\quad \la^{3a_3}x^{-1}=\al (f_1^{-1})^{a_1}(f_2^{-1})^{a_2}(\la^3f_3^{-1})^{a_3}\la^{-a_4},
$$
where $\al=\pm 1$ and it is enough to consider $-6< a_1,a_2< 6$, $0\leq a_3\leq 12$, $a_4=0$. To check whether 
$A^jT(x)\in\M_n(\bZ)$ (resp., $A^jT(\la^{3a_3}x^{-1})\in\M_n(\bZ)$), it is enough to consider $0\leq j\leq 6$. 
One checks that $T_1,T_2\in\Aut(G_A)$, and that
$T_3$ is not an endomorphism of $G_A$. Similarly, we check that 
$T_3^2\in\Aut(G_A)$. Thus, $\Aut(G_A)$ is a finitely generated abelian group and it is generated by 
$$
\Aut(G_A)=\{-1,A,T_1,T_2,T_3^2 \}.
$$
The approach in the proof of Proposition \ref{pr:decide} seems to produce a more complicated search, since the splitting field $L$ of $h$ has degree 
$5!=120$ with a complicated defining polynomial.

\sbr

We now compute $\vec{N}(X_A)$ and $\vec{N}(X_B)$. 
It is known that the unit group has rank $2$, the torsion has order $2$ and generated by $-1$, and 
$f_1$, 
$f_2$ are fundamental units \cite{sage}. 
One can check that $f_1(A)\not\in\M_5(\bZ)$ and $f_1^2\in\bZ[\la]$. Thus, $\vec{N}(X_A)$ is generated by 
$-I_5$ ($I_5$ is the $5\times 5$-identity matrix), $f_1^2(A)$, and $f_2(A)$. Explicitly, 
\begin{eqnarray*}
R_1&=&f_1^2(A)=14A^4 - 19A^3 - 50A^2 + 91A + 97I_5,\\ 
R_2&=&f_2(A)=2A^4 - 4A^3 - 4A^2 + 14A + 5I_5,
\end{eqnarray*}
 and 
$$
\vec{N}(X_A)=\{\pm(R_1)^j(R_2^k)\,\vert\, j,k\in\bZ\}.
$$ 
On the other hand, $f_1(B)\in\M_5(\bZ)$, hence $\vec{N}(X_B)\cong\cO_K^{\times}$. Explicitly, 
\begin{eqnarray*}
S_1&=&f_1(B)=\frac{2}{3}B^4-\frac{1}{3}B^3-\frac{11}{3}B^2+ \frac{10}{3}B + 9I_5,\\
S_2&=&f_2(B)=2B^4 - 4B^3 - 4B^2 + 14B + 5, 
\end{eqnarray*}
and
$$
\vec{N}(X_B)=\{\pm(S_1)^j(S_2^k)\,\vert\, j,k\in\bZ\}.
$$
\end{example}

\subsection{$2$-dimensional case}\label{ss:two0} 
For $n\in\bZ$, $n\ne\pm1 $, let $\rad(n)\in\bN$ be the product of all distinct prime divisors $p\in\bN$ of $n$.

\sbr 

If $n=2$, then there are three cases distinguished in \cite{s1}: \\

{\bf (a)} the characteristic polynomial $h_A\in\bZ[x]$ of $A$ is irreducible (equivalently, $A$ has no rational eigenvalues), \\

{\bf (b)} $h_A$ is reducible  (equivalently, $A$ has eigenvalues $\la_1,\la_2\in\bZ$), $\rad(\la_1)$ does not divide $\rad(\la_2)$, and $\rad(\la_2)$ does not divide $\rad(\la_1)$, \\

{\bf (c)} $h_A$ is reducible and every prime dividing one eigenvalue divides the other, {\it e.g.}, $\rad(\la_2)$ divides $\rad(\la_1)$ (denoted by 
$\rad(\la_2)\,\vert \rad(\la_1)$). \\

Case (a) is treated in Theorem \ref{th:endcomp}. 

\begin{rem}\label{rem:casebc0}
Note that if $n=2$, $h_A$ is reducible, and $\cP'\ne\emptyset$, then $\det A\ne\pm 1$, $A$ has distinct eigenvalues $\la_1,\la_2\in\bZ$, and hence $A$ is diagonalizable over $\bQ$. Moreover, there exists 
$S\in\operatorname{GL}_2(\bZ)$ such that $SAS^{-1}=M\La M^{-1}$, where
\bbe\label{eq:dim2}
\La=\left(\begin{matrix}
\la_1 & 0 \\
0 & \la_2 
\end{matrix}
\right), \quad 
M=\left(\begin{matrix}
1 & u \\
0 & v 
\end{matrix}
\right),
\,\, \la_1,\la_2,u,v\in \bZ,\,\,(u,v)=1,\,\,v\,\vert\,(\la_1- \la_2),
\ee
where $(u,v)=1$ means that $u,v$ are coprime 
\cite[Corollary A.2]{s1}. Since $S(G_{A})=G_{SAS^{-1}}$, {\it i.e.}, 
$G_A$, $G_{SAS^{-1}}$ are isomorphic, without loss of generality, we can assume that $A$ itself is upper-triangular and has the form $A=M\La M^{-1}$. Indeed, it is computable to find a matrix in $\operatorname{GL}_2(\bZ)$ that conjugates $A$ to a triangular form. 
One first diagonalizes $A$, and then multiplies an integer diagonalizing matrix on the left by an element of $\operatorname{GL}_2(\bZ)$ to transform it into upper-triangular form. This is a well-known argument (see {\it e.g.}, \cite{con1}).
\end{rem}

If $n=2$ and the characteristic polynomial $h_A\in\bZ[x]$ of $A$ is not irreducible, then Theorem \ref{th:2dimredb0} and Theorem \ref{th:2dimredc0} below describe $\End(G_A)$, $\Aut(G_A)$ completely. 

\begin{thm}[Theorem 3.8 \cite{s3}]\label{th:2dimredb0}
Assume $n=2$ and $A=M\La M^{-1}$, where $M,\La$ are given by $\eqref{eq:dim2}$. Assume case $(b)$, i.e., $\rad(\la_1)$ does not divide $\rad(\la_2)$, and $\rad(\la_2)$ does not divide $\rad(\la_1)$. Then $T\in\End(G_A)$ if and only if
\bbe\label{eq:T}
T=MXM^{-1},\quad X=\diag\left(\begin{matrix} x_1 & x_2\end{matrix}\right),
\ee
where $x_i\in\bZ[\la_i^{-1}]$, $i=1,2$, and $\frac{x_1-x_2}{v}\in\caR$. Moreover, $T\in\Aut(G_A)$ if and only if $T$ has the form \eqref{eq:T}, 
$x_i=\prod_{\text{prime }p\,\vert\rad(\la_i)}p^{k_p}$, $k_p\in\bZ$, $i=1,2$, and $\frac{x_1-x_2}{v}\in\caR$.
\end{thm}

Note that in the case $(b)$, by Theorem \ref{th:2dimredb0}, every $T\in\End(G_A)$ commutes with $A$. 
Recall that, up to multiplication by a power of $A$, $T$ is an integer matrix. Therefore, as a $\bZ[A,A^{-1}]$-module, $\End(G_A)$ is generated by elements of the centralizer of $A$ in $\M_2(\bZ)$.

\begin{thm}[Theorem 3.10 \cite{s3}]\label{th:2dimredc0}
Assume $n=2$, $A=M\La M^{-1}$, where $M,\La$ are given by $\eqref{eq:dim2}$, and $\cP'\ne\emptyset$. Assume case $(c)$, i.e., 
$\rad(\la_2)\,\vert \rad(\la_1)$. Then $T\in\End(G_A)$ if and only if 
\bbe\label{eq:Ttriang}
T=\left(\begin{matrix}
x & y \\
0 &  z 
\end{matrix}
\right)\in\M_2(\caR),\quad z\in\bZ[\la_2^{-1}].
\ee
Moreover, $T\in\Aut(G_A)$ if and only if $T$ has the form \eqref{eq:Ttriang}, where $x$ is a unit in $\caR$, $y\in\caR$, and
$z=\prod_{\text{prime }p\,\vert\rad(\la_2)}p^{k_p}$, $k_p\in\bZ$.
\end{thm}

Together with Theorem \ref{th:endcomp}, this fully resolves  the question of the computability of 
$\End(G_A)$, $\Aut(G_A)$ in the affirmative for $n=2$. 

\begin{cor}\label{cor:endcomp2}
Let $A\in\M_2(\bZ)$ be non-singular. 
Then $\End(G_A)$, $\Aut(G_A)$ are computable. 
\end{cor}
\begin{proof}
First, we check that $\frac{x_1-x_2}{v}\in\caR$ is computable. Indeed, by the definition of $\caR$, 
without loss of generality, we can assume $(v,\det A)=1$. In other words, we write $v=v_1v_2$, $v_1,v_2\in\bZ$ with 
$\rad(v_1)$ dividing $\det A$ and $(v_2,\det A)=1$. Then $\frac{x_1-x_2}{v}\in\caR$ if and only if $\frac{x_1-x_2}{v_2}\in\caR$. 
In the case of $\End(G_A)$, $x_i=n_i\la_i^{m_i}$, $n_i,m_i\in\bZ$, $i=1,2$. 
It is enough to check for finitely many $n_i$ modulo $v$ and $m_i$ such that $\abs{m_i}$ is bounded by the order of $\la_i$ modulo $v$.
In the case of $\Aut(G_A)$, any $p$ dividing $\rad(\la_2)$ has a finite order modulo $v$. Thus, to check the condition, it is enough to consider finitely many powers $k_p$ such that  
$\abs{k_p}$ is bounded by the order of $p$ modulo $v$. Everything else follows from Theorem \ref{th:endcomp}, Theorem \ref{th:2dimredb0}, and Theorem \ref{th:2dimredc0}. 
\end{proof}

\begin{example}
Let
$$
A=\begin{pmatrix}
100 & 110 \\
120 & 162
\end{pmatrix},
$$
$h_A=x^2-262x+3000$ with roots $\la_1=2^2\cdot 3=12$ and $\la_2=2\cdot 5^3=250$, $\rad\la_1=6$, $\rad\la_2=10$. Thus, 
Theorem \ref{th:2dimredb0} applies, and we use it together with the proof of Corollary \ref{cor:endcomp2} to compute $\End(G_A)$ and
$\Aut(G_A)$.
We have that $\caR=\{2^i3^j5^k\,\vert\,i,j,k\in\bZ\}$, 
$$
S=\begin{pmatrix}
5 & -1 \\
-4 & 1
\end{pmatrix},\quad
M=\begin{pmatrix}
1 & 26 \\
0 & 119
\end{pmatrix},\quad
S^{-1}AS = M\begin{pmatrix}
12 & 0 \\
0 & 250
\end{pmatrix}M^{-1},
$$
where $S\in\GL_2(\bZ)$ and 
$v=119=7\cdot 17$.  We first find $\End(G_A)$. Clearly, from Theorem \ref{th:2dimredb0}, $\End(G_A)$ is commutative, so it has a structure of a $\bZ[A,A^{-1}]$-module. By Theorem \ref{th:2dimredb0}, for $T\in\End(G_A)$ we have that $S^{-1}TS=MXM^{-1}$, $X=\diag\left(\begin{matrix} x_1 & x_2\end{matrix}\right)$, where $x_1\in\bZ[\frac{1}{12}]$, $x_2\in\bZ[\frac{1}{250}]$, and $\frac{x_1-x_2}{119}\in\caR$. 
As a $\bZ[A,A^{-1}]$-module, $\End(G_A)$ is generated by the $2\times 2$-identity matrix and 
$$
S_1= 
\begin{pmatrix}
44 & 55 \\
60 & 75
\end{pmatrix},
\quad 
S_2=
\begin{pmatrix}
75 & -55 \\
-60 & 44
\end{pmatrix}.
$$
To find 
$\Aut(G_A)$, by Theorem \ref{th:2dimredb0}, we write $x_1=2^a3^b$, $x_2=2^c5^d$ and find all $a,b,c,d\in\bZ$ such that 
$(x_1-x_2)/119\in\caR$. That gives $b\equiv 5d-14(a-c)\,\text{mod} \,48$, where $a,c,d\in\bZ$ are arbitrary.
\end{example}

\section{Computability of  $\Isom(G_A,G_B)$} 
%
%
%
%
In this section, we assume that the characteristic polynomial $h_A\in\bZ[t]$ of $A\in\M_n(\bZ)$ is irreducible and there exists a prime $p\in\bN$ such that $n$ and $t_p$ are coprime 
(denoted by $(n,t_p)=1$), where $t_p$ is the multiplicity of zero in the reduction of 
${h}_A$ modulo $p$. Let $B\in\M_n(\bZ)$ be non-singular. We show that there is an algorithm to find all (group) isomorphisms between $G_A$ and $G_B$ or to show that $G_A$, $G_B$ are not isomorphic. This demonstrates that given the assumptions stated above, the question of whether $G_A$ and $G_B$ are isomorphic is decidable and the set of all isomorphisms between $G_A$ and $G_B$ is computable. By decidable, we mean that there exists an algorithm that terminates in finitely many steps. We do not, however, make any claims about the efficiency or practicality of this algorithm. 

\begin{rem}\label{rem:easycases}
If $A\in\GL_n(\bZ)$, then $G_A\cong G_B$ if and only if $B\in\GL_n(\bZ)$. Moreover, in that case, 
$G_A=G_B=\bZ^n$ and $\Isom(G_A,G_B)=\GL_n(\bZ)$. It follows from \cite[Lemma 3.6, Lemma 3.11]{s1} that if $A\not\in\GL_n(\bZ)$ and $\cP'(A)=\emptyset$, then $G_A\cong G_B$ if and only if $\det A$, $\det B$ have the same prime divisors. Moreover, in that case, 
$\cP'(B)=\emptyset$ and $\Isom(G_A,G_B)=\GL_n(\caR)$. Clearly, the conditions are computable, so   for the rest of the section we assume that 
$A\not\in\GL_n(\bZ)$ and $\cP'(A)\ne\emptyset$. 
\end{rem}
\begin{thm}\label{th:compirr}
Let $A,B\in\M_n(\bZ)$ be non-singular. Assume that the characteristic polynomial $h_A\in\bZ[t]$ of $A$ is irreducible and there exists $t_p$ such that $(n,t_p)=1$. Then the existence of an isomorphism between $G_A$ and $G_B$ is decidable and the set of all  
isomorphisms between $G_A$ and $G_B$ is computable.
\end{thm}

\begin{proof}
Clearly, if there exists an isomorphism $T_0:G_A\rar G_B$, then the set of all the isomorphisms from $G_A$ to $G_B$ has the form
$$
\{ T_0\phi \,\vert\,\phi\in\Aut(G_A) \}
$$
and by Theorem \ref{th:endcomp}, the group $\Aut(G_A)$ is computable. 
Thus, it suffices to determine whether $G_A$ and $G_B$ are isomorphic, and if so, to find an explicit isomorphism. However, in our view, the complexity of finding a single isomorphism is comparable to that of finding all of them.

\sbr

Assume $T:G_A\rar G_B$ is an isomorphism.
By \cite[Definition 5.1, Proposition 5.7]{s2},  there exist eigenvalues $\la,\mu\in\overline{\bQ}$ of $A$, $B$, respectively, such that 
$K=\bQ(\la)=\bQ(\mu)$ and $\la,\mu$ have the same prime ideal divisors in the ring of integers $\cO_K$ of $K$. 
Both conditions are computable and implemented into \cite{sage} (see {\it e.g.}, \cite[Algorithm 4.5.5]{cohen}, \cite[Algorithm $A_{10}$]{gr}). 
Multiplying an eigenvector ${\bf u}\in K^n$ of $A$ corresponding to $\la$ by an appropriate integer, without loss of generality, we can assume ${\bf u}\in(\bZ[\la])^n$, where $\bZ[\la]\subseteq\cO_K$. 
Since $h_A$ is irreducible, there exist (field) embeddings $\{\sg_1,\ldots,\sg_n\}$ of $K$ into $\overline{\bQ}$, 
$\sg_1=\id$, such that $\{\sg_1(\la),\ldots,\sg_n(\la)\}$ are all eigenvalues of $A$. Thus, $A=M\La M^{-1}$, where
$$
\La=\diag\left(\begin{matrix} \sg_1(\la)& \ldots & \sg_n(\la)\end{matrix}\right),\quad M=\left(\begin{matrix} \sg_1({\bf u}) & \ldots & \sg_n({\bf u}\end{matrix})\right)
$$
with each $\sg_i({\bf u})$ written as a column, $i\in\{1,\ldots,n\}$. Analogously, we choose 
an eigenvector ${\bf v}\in(\bZ[\mu])^n$ of $B$ corresponding to $\mu$, $\bZ[\mu]\subseteq\cO_K$. Then $B=N\Ga N^{-1}$, where
$$
\Ga=\diag\left(\begin{matrix} \sg_1(\mu)& \ldots & \sg_n(\mu)\end{matrix}\right),\quad N=\left(\begin{matrix} \sg_1({\bf v}) & \ldots & \sg_n({\bf v}\end{matrix})\right).
$$
By \cite[Lemma 3.1]{s1}, \cite[Definition 5.1, Proposition 5.7]{s2}, if $T:G_A\to G_B$ is an isomorphism, then $T\in\GL_n(\bQ)$ and there exists $x\in K-\{0\}$ such that
\bbe\label{eq:t1}
T=NXM^{-1},\quad X=\diag\left(\begin{matrix} \sg_1(x)&\ldots &\sg_n(x)\end{matrix}\right),\quad T=T(x).
\ee
Let $\om_1,\ldots,\om_n\in\cO_K$ be an integral basis of $\cO_K$, and let
$\Om=(\sg_j(\om_i))_{i,j}$. Since ${\bf u},{\bf v}\in(\cO_K)^n$ by assumption, there exist non-singular $L_1,L_2\in\M_n(\bZ)$ such that
$M=L_1\Om$, $N=L_2\Om$. 
Let
\begin{eqnarray*}
l_1&=&\abs{\det L_1}=[\cO_K:\bZ[ {\bf u}]], \\ 
l_2&=&\abs{\det L_2}=[\cO_K:\bZ[ {\bf v}]], \\
t&=&\min\{[\cO_K:\bZ[\la]],[\cO_K:\bZ[ \mu]] \}.
\end{eqnarray*}
By the definition of $G_A$, $T\in\M_n(\bQ)$ is a homomorphism from 
 $G_A$ to $G_B$ if and only if 
for any $k\in\bN\cup\{0\}$ there exists $l_k\in\bN\cup\{0\}$ with
\bbe\label{eq:main1}
B^{l_k}TA^{-k}\in \operatorname{M}_n(\bZ).
\ee
In particular, $S=B^{l_0}T\in\M_n(\bZ)$. Since $T\in\Isom(G_A,G_B)$ if and only if $S\in\Isom(G_A,G_B)$, without loss of generality, we can assume that $T\in\M_n(\bZ)$.  
Let $F\subset\overline{\bQ}$ be a number field that contains all the eigenvalues of $A$, so that $M\in\M_n(\cO_F)$. 
By \cite[p. 4, Theorem 2]{nt}, there exists a finite extension $L\subset\overline{\bQ}$ of $F$ and 
$P\in\GL_n(\cO_L)$ such that $PM$ is upper-triangular. Thus, from  \eqref{eq:t1} 
\bbe\label{eq:last}
PM XM^{-1}P^{-1}=P(MN^{-1})TP^{-1}. 
\ee 
By above, $MN^{-1}=L_1L_2^{-1}\in\frac{1}{l_2}\M_n(\bZ)$, so that the entries of the matrix on the right in \eqref{eq:last} have denominators dividing $l_2$, since $T\in\M_n(\bZ)$ and $P\in\GL_n(\cO_L)$. On the other hand, the $(1,1)$-element of the matrix on the left in \eqref{eq:last} equals
$x$, since $PM$ is upper triangular and  $X$ is diagonal. Since $x\in K$, this implies that there exists $y\in\cO_K$ such that 
\bbe\label{eq:x}
x=\frac{y}{l_2}, \quad y\in\cO_K-\{0\}.
\ee
Since $\la,\mu$ have the same prime divisors, for each $k\in\bN$ there exists $l_k\in\bN$ 
big enough such that $\mu^{l_k}\la^{-k}\in\cO_K$. Thus, in \eqref{eq:main1}, for the rest of the proof,
 without loss of generality, we can assume that 
$\mu^{l_k}\la^{-k}\in\cO_K$ for each $k\in\bN$. These imply that the denominators of entries of $B^{l_k}TA^{-k}$ are bounded, since $B^{l_k}TA^{-k}=T(\mu^{\l_k}\la^{-k}x)$. Hence, verifying \eqref{eq:main1} for all $k\in\bN$ and all $x$ given by 
\eqref{eq:x} reduces to checking finitely many congruences, which is computable. We now make this statement more precise.
Note that for any $x\in tl_1\cO_K$, \eqref{eq:main1} holds for $T=T(x)$.
Indeed, if $x\in tl_1\cO_K$, then for any $k\in\bN$ and $l_k\in\bN$ such that $\mu^{\l_k}\la^{-k}\in\cO_K$, we have that
$\mu^{\l_k}\la^{-k}x=tl_1w$
 for some $w\in\cO_K$. Let $t=[\cO_K:\bZ[\la]]$. Then  
$tw\in\bZ[\la]$, that is $tw=p(\la)$, a polynomial with integer coefficients in $\la$, and 
 $$M\diag\left(\begin{matrix} \sg_1(tw)&\ldots &\sg_n(tw)\end{matrix}\right) M^{-1}=p(A)\in\M_n(\bZ).$$ Thus, 
\begin{eqnarray*}
&& B^{l_k}TA^{-k}=T(\mu^{\l_k}\la^{-k}x)=\\ 
&&=(l_1NM^{-1})
M\diag\left(\begin{matrix} \sg_1(tw)&\ldots &\sg_n(tw)\end{matrix}\right) M^{-1}\in\M_n(\bZ)
\end{eqnarray*}
as a product of two integer matrices. Similarly, if $t=[\cO_K:\bZ[\mu]]$, then  
$tw\in\bZ[\mu]$ and
\begin{eqnarray*}
&& B^{l_k}TA^{-k}=T(\mu^{\l_k}\la^{-k}x)=\\ 
&&=
(N\diag\left(\begin{matrix} \sg_1(tw)&\ldots &\sg_n(tw)\end{matrix}\right) N^{-1})(l_1NM^{-1})\in\M_n(\bZ).
\end{eqnarray*}
Since $x=y/l_2$ for some $y\in\cO_K$, we conclude that \eqref{eq:main1} holds for $T=T(y/l_2)$ if and only if it holds for $y$ modulo $tl_1l_2\cO_K$.
Finally, note that the set $\{\mu^{l_k}\la^{-k}\in\cO_K\,\vert\, k,l_k\in\bN \}$ is finite modulo $tl_1l_2$, since
$\cO_K/(tl_1l_2\cO_K)$ is finite. In other words, it follows that for $l_k$ such that $\mu^{l_k}\la^{-k}\in\cO_K$, 
the denominators of entries of $B^{l_k}TA^{-k}\in\GL_n(\bQ)$ divide $tl_1l_2$, that is,
each matrix $tl_1l_2B^{l_k}TA^{-k}\in\M_n(\bZ)$ has integer coefficients. Thus, 
\eqref{eq:main1} corresponds to finitely many congruences in $\bZ$:
\bbe\label{eq:main1red}
tl_1l_2B^{l_k}TA^{-k}\equiv 0\text{ mod }tl_1l_2,
\ee
where $0$ on the right is the zero $n\times n$-matrix. Similarly, by definition, $T^{-1}:G_B\rar G_A$ is a homomorphism if and only if
for any $i\in\bN\cup\{0\}$ there exists $j_i\in\bN\cup\{0\}$ with
\bbe\label{eq:main2}
A^{j_i}T^{-1}B^{-i}\in \M_n(\bZ).
\ee
This is equivalent to checking that for any $i\in\bN\cup\{0\}$ and any $j_i$ such that $\la^{j_i}\mu^{-i}\in\cO_K$: 
\bbe\label{eq:main2red}
tl_1l_2 A^{j_i}T^{-1}B^{-i}\equiv 0\text{ mod }tl_1l_2,
\ee
where $A^{j_i}T^{-1}B^{-i}$ is a matrix with rational entries, and $tl_1l_2 A^{j_i}T^{-1}B^{-i}$ is a matrix with integer entries. 
Note that  
$T^{-1}(x^{-1})=MX^{-1}N^{-1}$. 
As in the case of $T$ above, if $T^{-1}(G_B)\subseteq G_A$ (equivalently, $T^{-1}:G_B\rar G_A$ is a homomorphism), then there exists $\la^{j}$, $j\in\bN\cup\{0\}$, and $z\in\cO_K$ such that
$\frac{\la^j}{x}=\frac{z}{l_1}$. This implies 
\bbe\label{eq:yz}
yz=\la^jl_1l_2.
\ee
The uniqueness of prime ideal factorization in $\cO_K$ allows us to identify finitely many candidates for $y$.
Namely, from \eqref{eq:yz}, $(y)=y\cO_K=\fa_1\fa_2$, where $\fa_1\subseteq\cO_K$ is an ideal dividing $\la^j$, and $\fa_2\subseteq\cO_K$ is an ideal dividing $l_1l_2$. Observe that the norm of $\fa_2$ is bounded. Let 
$(\la)=\fp_1\cdots\fp_s$, $(l_1l_2)=\fq_1\cdots\fq_r$ for some (not necessarily distinct) prime ideals 
$\fp_1,\ldots,\fp_s$, $\fq_1,\ldots,\fq_r$ of $\cO_K$. It is known that every ideal has a finite order in the class group of $\cO_K$.  For each $i$, let $n_i\in\bN$ be the order of $\fp_i$ in the class group of $\cO_K$, that is $\fp_i^{n_i}=(a_i)$ is principal for some $a_i\in\cO_K$, $1\leq i\leq s$. Denote 
$$
J=\{\fp_i^{j},\fq_1,\ldots,\fq_r\,\vert\, 0\leq j< n_i,\,\, 1\leq i\leq s\},
$$ 
a finite set. 
We now check which ones out of finitely many ideals 
 dividing the product of all ideals in $J$ are principal. 
That will determine $y$ up to multiplication by a unit in $\cO_K^{\times}$. It is known that $\cO_K^{\times}$ is a finitely generated group. As was explained above, it is enough to check \eqref{eq:main1red} and \eqref{eq:main2red} for 
$y$ and $z$ modulo $tl_1l_2\cO_K$. Therefore, assuming the necessary conditions for $G_A\cong G_B$ described at the beginning of the proof are satisfied, finding all isomorphisms between $G_A$ and $G_B$ reduces to checking finitely many congruences for finitely many candidate pairs $y,z\in\cO_K$. 

\sbr

Finally, it is known that all the ingredients in the proof are computable, such as computing an integral basis of $\cO_K$, starting from $1,\la,\ldots,\la^{n-1}$, as well as indices $l_1,l_2$ ({\it e.g.}, Algorithm $A_6$ \cite{gr}), prime decompositions of ideals (see {\it e.g.}, \cite[Algorithm $A_{10}$]{gr}), deciding if a product of prime ideals is principal and if yes, producing a generator (see {\it e.g.}, \cite[Algorithm 6.5.10]{cohen})), computing the class number of $K$ (see {\it e.g.}, \cite[Algorithm $A_{12}$]{gr}), finding generating sets of $\cO_K^{\times}$ (see {\it e.g.}, \cite[Algorithm $A_8$]{gr}), and solving systems of congruences either by applying the Chinese Remainder Theorem when the moduli are coprime, or by using more general methods such as the extended Euclidean algorithm and verifying compatibility when the moduli are not coprime. 
\end{proof}

\begin{rem}\label{rem:compirr}
Let $A,B\in\M_n(\bZ)$ be non-singular, $\cP'(A)\ne\emptyset$. Assume in addition that $A$, $B$ are conjugate in $\GL_n(\bQ)$, {\it i.e.}, there exists $P\in\GL_n(\bQ)$ such that $B=PAP^{-1}$. 

\sbr

If $n=2$, the characteristic polynomial of $A$ is irreducible or 
$A$ has eigenvalues $\la_1,\la_2\in\bZ$ such that $\rad\la_i$ does not divide $\rad\la_j$, $i,j\in\{1,2\}$, $i\ne j$ (cases 
(a), (b) in Section \ref{ss:two0} above), then for any isomorphism $T$ from $G_A$ to $G_B$ there exists $l\in\bN\cup\{0\}$ such that
$B^lT\in\M_2(\bZ)$, and the transpose $(B^lT)^t$ of $B^lT$ is an epimorphism of dimension groups from $\mathscr{D}_{A^t}$ to $\mathscr{D}_{B^t}$ as defined in \cite{kr}. Conversely, the transpose of any epimorphism from $\mathscr{D}_{A^t}$ to $\mathscr{D}_{B^t}$ is an isomorphism from $G_A$ to $G_B$ (by \cite[Lemma 6.5]{s1}). Analogously, the same result holds 
 for an arbitrary $n$ such that there exists $t_p=t_p(A)$ satisfying $(n,t_p)=1$ and $h_A$ is irreducible.
%
\end{rem}

In the following examples, we illustrate how the procedure outlined in the preceding proof can be applied, using the same notation introduced there. Number field computations were performed using \cite{db} and \cite{sage}.

\begin{example}
Let  
$$
A=\left(\begin{matrix}
0 & -6 \\
1 & 1 
\end{matrix}
\right),\,\,
B=\left(\begin{matrix}
0 & -3 \\
2 & 1 
\end{matrix}
\right). 
$$
Both $A$ and $B$ have the same characteristic polynomial $h=x^2-x+6$ with root $\la=\frac{1+\sqrt{-23}}{2}$.
One can check that $A$, $B$ are $\GL_2(\bQ)$-conjugate but not $\GL_2(\bZ)$-conjugate to each other. 
In \cite[Example 7.2]{s1}, we showed that $G_A\cong G_B$ by finding one isomorphism between $G_A$ and $G_B$. 
Now, we find all isomorphisms between $G_A$ and $G_B$. In the notation of the proof of Theorem \ref{th:compirr}, 
\begin{eqnarray*}
&& \La=\Ga=\left(\begin{matrix}
\la & 0 \\
0 & \sg(\la) 
\end{matrix}
\right),\,\,
M=\left(\begin{matrix}
-\sg(\la) & -\la \\
1 & 1 
\end{matrix}
\right),\\
&& N=\left(\begin{matrix}
-\sg(\la) & -\la \\
2 & 2 
\end{matrix}
\right), 
\end{eqnarray*}
where $\sg(\la)=\frac{1-\sqrt{-23}}{2}$ is another eigenvalue of $A$ and
$
A=M\La M^{-1},\,\,
B=N\La N^{-1}. 
$
Let $K=\bQ(\la)$. It is known that $\cO_K=\bZ[\la]$, so that $t=1$, $l_1=1$, and $l_2=2$. 
Thus, from \eqref{eq:yz}, we have that $yz=2\la^j$. We now look at prime decompositions of ideals $(\la)$ and $(2)$.
We have $(\la)=\fp_1\fq_1$, 
$(2)=\fp_1\fp_2$, $(3)=\fq_1\fq_2$, 
where $\fp_i,\fq_i$ are prime ideals of $\cO_K$ of order $3$, $i=1,2$. This implies 
$(y)=\fp_1^{i_1}\fq_1^{i_2}\fp_2^{i_3}$, where $i_3=0,1$, and $i_1,i_2\in\bN\cup\{0\}$, and
$J=\{\fp_1^{j_1},\fq_1^{j_2},\fp_1,\fp_2\,\vert\, 0\leq j_1\leq 2,\, 0\leq j_2\leq 2 \}$. Thus, it is enough to check if 
$\fp_1^{i_1}\fq_1^{i_2}\fp_2^{i_3}$ is principal for $0\leq i_1\leq 2$, $0\leq i_2\leq 2$, $0\leq i_3\leq 1$. 
Let 
$\fp_1^3=(a)$, $\fq_1^3=(b)$, and $\fq_1^2\fp_2=(c)$, where $a=\la-2$, $b=2\la-3$, and $c=2b/\la=\la+3$. 
One checks all the possibilities: $(1)$, $(c)$, $(2)$, $(\la)$, $(2\la)$, and $(\la^2)$. We compute that $\cO_K^{\times}=\{\pm 1\}$ is finite of order $2$. Thus, $y$ equals one of the following 
$$
\pm a^kb^s,\, \pm a^kb^sc,\, \pm 2a^kb^s, \,\pm a^kb^s\la,\, \pm 2a^kb^s\la,\, \pm a^kb^s\la^2,
$$
where $k,s\in\bN\cup\{0\}$. Note that $a^kb^s=\pm\la^k$ or $a^kb^s=\pm\la^kb^{s-k}$ or $a^kb^s=\pm\la^sa^{k-s}$, 
since $(\la^3)=\fp_1^3\fq_1^3=(ab)$. Thus, $y$ equals one of the following
\bbe\label{eq:ylist}
\pm\la^j a^kb^s,\, \pm 2\la^j a^kb^s,\, \pm c\la^j a^kb^s,
\ee
where $j,k,s\in\bN\cup\{0\}$ and $ks=0$. 
One checks \eqref{eq:main1} for $T=T(y/2)$ for any $y$ in \eqref{eq:ylist} modulo
$2\cO_K$ and up to multiplication by an integer power of $\la$. Note that $a\equiv \la\,\text{mod}\, 2\cO_K$,
$b\equiv 1\,\text{mod}\, 2\cO_K$, and $\la c\equiv 0\,\text{mod}\, 2\cO_K$. Hence, it is enough to check 
\eqref{eq:main1red} for $y=1$ and $T=T(1/2)=\frac{1}{2}NM^{-1}$. 
Here, one checks that
$\frac{1}{2}NM^{-1}A\in\M_n(\bZ)$, so that $B^{k+1}TA^{-k}=\frac{1}{2}NM^{-1}A\in\M_n(\bZ)$ for any $k\in\bN\cup\{0\}$ and hence \eqref{eq:main1red} holds for any $y$ in \eqref{eq:ylist} and $T=T(y/2)$.
 Analogously, we have the same possibilities \eqref{eq:ylist} for $z$, for which we check \eqref{eq:main2red} 
 with $T^{-1}=T^{-1}(z)$. 
As with $y$, for the purpose of checking \eqref{eq:main2red}, $z$ is defined modulo $2\cO_K$ and up to a multiplication by an integer power of $\la$, hence it is enough to check
that \eqref{eq:main2red} holds for $z=1$ and $T^{-1}=T(1)^{-1}=MN^{-1}$. We compute
$$
2A^{j_i}T^{-1}B^{-i}=2MN^{-1}B^{j_i-i}=\left(\begin{matrix}
2 & 0 \\
0 & 1 
\end{matrix}
\right)B^{j_i-i},
$$
and we check whether
\bbe\label{eq:bj}
\left(\begin{matrix}
2 & 0 \\
0 & 1 
\end{matrix}
\right)B^j\equiv 0\text{ mod }2.
\ee
We have that
$$
B\equiv \left(\begin{matrix}
0 & 1 \\
0 & 1 
\end{matrix}
\right) \text{ mod }2,\quad
\left(\begin{matrix}
2 & 0 \\
0 & 1 
\end{matrix}
\right)\equiv\left(\begin{matrix}
0 & 0 \\
0 & 1 
\end{matrix}
\right)\text{ mod }2,\quad
\left(\begin{matrix}
2 & 0 \\
0 & 1 
\end{matrix}
\right)B^j\equiv
\left(\begin{matrix}
0 & 0 \\
0 & 1 
\end{matrix}
\right)\text{ mod }2 
$$
for any $j\in\bN\cup\{0\}$, 
so that \eqref{eq:bj} does not hold for $z=1$. This implies that 
$z=2\la^j a^kb^s$ or $z=\pm c\la^j a^kb^s$ by \eqref{eq:ylist}. Since $yz=2\la^j$, one concludes that $y=\pm a^kb^s\la^i$ and hence 
$$
x=\frac{y}{l_2}=\frac{y}{2}=\pm\frac{a^kb^s\la^i}{2},\quad  k,s\in\bN\cup\{0\},\,ks=0,\,i\in\bZ. 
$$
That gives all (infinitely many) isomorphisms  $T=T(x):G_A\rar G_B$ and,  
in particular, $G_A\cong G_B$. The result agrees with our previous calculation in \cite[Example 7.2]{s1}, 
where we found just one isomorphism $T=T(1/2):G_A\rar G_B$. Clearly, if there is one isomorphism 
$T:G_A\rar G_B$, then there are infinitely many, since $B^jTA^i$ is also an isomorphism for any $i,j\in\bZ$.
Nonetheless, in this example, we are able to find other families of isomorphisms generated by $\frac{a^k}{2}$
and $\frac{b^s}{2}$. 
\end{example}

\begin{example}\label{ex:different} 
Let 
\[
A=\begin{pmatrix}
0 & 0 & 7 \\
1 & 0 & -11 \\
0 & 1 & 6
\end{pmatrix}, \quad
B=\begin{pmatrix}
0 & 0 & 7 \\
1 & 0 & 9 \\
0 & 1 & 2
\end{pmatrix}
\]
be companion matrices of $h_A=x^3-6x^2+11x-7$, $h_B=x^3-2x^2-9x-7$, respectively. Both $h_A$, $h_B$ are irreducible over $\bQ$, $\cP(A)=\cP(B)=\cP'(A)=\cP'(B)=\{7\}$, and $t_7(A)=t_7(B)=1$. Thus, Theorem \ref{th:compirr} applies. 
One can check that there exist a root $\la$ of $h_A$ and a root $\mu$ of $h_B$ such that $K=\bQ(\la)=\bQ(\mu)$, and
$\la$, $\mu$ have the same prime ideal divisors in $\cO_K$. Let
\[
{\bf u}(\lambda)=\begin{pmatrix}
\lambda^{2}-6\lambda+11\\[2pt]
\lambda-6\\[2pt]
1
\end{pmatrix},\quad 
{\bf v}(\mu)=\begin{pmatrix}
\mu^{2}-2\mu-9\\[2pt]
\mu-2\\[2pt]
1
\end{pmatrix}
\]
be eigenvectors of $A$, $B$ corresponding to eigenvalues $\la$, $\mu$, respectively. Then, in the notation of Theorem \ref{th:compirr}, $\cO_K=\bZ[\la]$, $NM^{-1}\in\M_n(\bZ)$, $\det NM^{-1}=5$, $t=l_1=1$, $l_2=5$, $yz=5\la^i$, $\cO_K^{\times}$ is generated by $-1$ and 
$u=\la-2$, the ideal $(\la)$ is prime and $(5)=\fp_1\fp_2$, where 
$\fp_1=(\al)$, $\fp_2=(\be)$, $\al=\al(\la)=-2\la^2+9\la-11$, $\be=\be(\la)=-\la^2+5\la-4$. Note that $A$ and $A-2I$ are invertible modulo $5$, where $I$ is the $3\times 3$-identity matrix. We check for $y=\la^iu^j\al^r\be^s$, where $i,j\in\bZ$, $r,s\in\{0,1\}$, whether 
$$
T(x)=\frac{1}{5}T(y)=\frac{1}{5}NM^{-1}A^i(A-2I)^j\al(A)^r\be(A)^s
$$ 
is integer. It turns out that $T(x)$ is integer if and only if $r=s=1$. In that case, we check whether $A^kT^{-1}(z)$ is integer when $z=u^{l}$, $k,l\in\bZ$. Note that
$$
A^kT^{-1}(z)=A^{k}(A-2I)^{l}MN^{-1},
$$ 
and some of the entries of $MN^{-1}$ have $5$ in denominators. Thus, $5MN^{-1}$ is integer, and 
we check whether 
$5A^{k}(A-2I)^{l}MN^{-1}\equiv 0\text{ mod }5$. Recall that $A$, $A-2I$ are invertible modulo $5$, so one should have 
$5MN^{-1}\equiv 0\text{ mod }5$, which is not true. Therefore, $G_A$ and $G_B$ are not isomorphic. 
\end{example}

\subsection{$2$-dimensional case}\label{ss:two02} 
In this section, we give a complete solution to the decidability problem of whether two groups $G_A$ and $G_B$  
 are isomorphic, and to the computability of 
$\Isom(G_A,G_B)$, in the two-dimensional case.

\begin{prop}\label{pr:comp2}
Let $A,B\in\M_2(\bZ)$ be non-singular. The existence of an isomorphism between $G_A$ and $G_B$ is decidable and the set of all  
isomorphisms between $G_A$ and $G_B$ is computable.
\end{prop}

\begin{proof}
By Remark \ref{rem:easycases}, we assume that 
$A\not\in\GL_n(\bZ)$ and $\cP'(A)\ne\emptyset$  for the rest of the proof.
As in Section \ref{ss:two0}, 
if $n=2$, then there are three cases. 
In the case when the characteristic polynomial $h_A\in\bZ[x]$ of $A$ is irreducible (equivalently, $A$ has no rational eigenvalues), the result follows from Theorem \ref{th:compirr}. 

\sbr

Assume the case when $h_A$ is reducible  (equivalently, $A$ has eigenvalues $\la_1,\la_2\in\bZ$) and every prime dividing one eigenvalue divides the other, {\it e.g.}, 
$\rad(\la_2)\,\vert \rad(\la_1)$. By \cite[Proposition 5.1]{s1}, we have that $G_A\cong G_B$ if and only if 
$B$ has eigenvalues $\mu_1,\mu_2\in\bZ$ such that $\rad\mu_i=\rad\la_i$, $i=1,2$. The condition is computable, hence 
whether $G_A\cong G_B$ is decidable. We now discuss how to find all isomorphisms from $G_A$ to $G_B$. By \cite[Remark 4.2, Proposition 5.1]{s1}, there exist $S,P\in\operatorname{GL}_2(\bZ)$ such that 
$$
A=SM\left(\begin{matrix}
\la_1 & 0 \\
0 & \la_2 
\end{matrix}
\right)M^{-1}S^{-1},\,\,
M=\left(\begin{matrix}
1 & u \\
0 & v 
\end{matrix}
\right),
$$
$$
B=PN\left(\begin{matrix}
\mu_1 & 0 \\
0 & \mu_2 
\end{matrix}
\right)N^{-1}P^{-1},\,\,
N=\left(\begin{matrix}
1 & x \\
0 & y 
\end{matrix}
\right),
$$
where $\la_1,\la_2,u,v\in \bZ$, $v\,\vert\,(\la_1- \la_2)$, $(u,v)=1$, 
$\mu_1,\mu_2,x,y\in \bZ$, $y\,\vert\,(\mu_1- \mu_2)$,
$(x,y)=1$, $\rad\la_i=\rad\mu_i$, $i=1,2$,
\bbe\label{eq:T}
T=T(\nu_i)=PN\left(\begin{matrix}
\nu_1 & \nu_2 \\
0 & \nu_3
\end{matrix}
\right)M^{-1}S^{-1}\in\GL_2(\caR),\quad \nu_1,\nu_3\in\bQ-\{0\},\,\nu_2\in\bQ.  
\ee
(Note that the above decompositions of $A,B$ are computable by Remark \ref{rem:casebc0}.)
Since $B^iT\in\M_n(\bZ)$ for some $i\in\bN\cup\{0\}$, without loss of generality, we can assume that $T$ itself is integer. Moreover, there exists $j\in\bN\cup\{0\}$ such that
$A^jT^{-1}\in\M_n(\bZ)$. These imply that $\nu_1\in\bZ$ has the same prime divisors as $\la_1$, $\nu_3=kv/y$, 
$k\in\bZ$, $k$ has the same prime divisors as $\la_2$, and 
\bbe\label{eq:main22}
-\frac{u}{v}\nu_1+\frac{x}{y}k+\frac{\nu_2}{v}\in\bZ.
\ee
Furthermore, any $T$ as above is an isomorphism from $G_A$ to $G_B$.
Let 
\begin{eqnarray}\label{eq:triple}
&& \nu_1=\prod_{\text{prime }p\,\vert\la_1}p^{r_p},\quad k=\prod_{\text{prime }p\,\vert\la_2}p^{s_p},\quad\forall r_p,s_p\in\bN\cup\{0\},\\
&& \nu_3=\frac{kv}{y},\quad \nu_2=mv-x\nu_3+u\nu_1,\quad m\in\bZ, \nonumber
\end{eqnarray}
by \eqref{eq:main22}. Then $T\in\Isom(G_A,G_B)$ if and only if $T=B^jT(\nu_i)$, where $j\in\bZ$,
$T(\nu_i)$ is given by \eqref{eq:T}, and $\nu_1,\nu_2,\nu_3$ are given by \eqref{eq:triple}. 

\sbr

Assume the case when $h_A$ is reducible, $\rad(\la_1)$ does not divide $\rad(\la_2)$, and $\rad(\la_2)$ does not divide $\rad(\la_1)$. The result follows from \cite[Proposition 6.9]{s1}, where we addressed the existence of at least one isomorphism from $G_A$ to $G_B$. Now, we are interested in finding all of them. By \cite[Proposition 6.9]{s1}, 
$T\in\Isom(G_A,G_B)$ if and only if $T=B^jT(\nu_i)$, where $j\in\bZ$,
$T(\nu_i)$ is given by \eqref{eq:T}, and $\nu_1$, $\nu_3$ are given by \eqref{eq:triple}, $\nu_2=0$, and \eqref{eq:main22} holds, {\it i.e.}, 
%
\bbe\label{eq:main2}
-\frac{u}{v}\nu_1+\frac{x}{y}k\in\bZ.
\ee
Note that \eqref{eq:main2} gives a system of finitely many congruences for finitely many powers $\{r_p\}_p$, $\{s_q\}_q$. Solving the system is computable and therefore whether $G_A$, $G_B$ are isomorphic is decidable. 
\end{proof}

\begin{example} 
Let 
$
A=\begin{pmatrix}
89 & -69 \\
35 & -15
\end{pmatrix}
$, 
$
B=\begin{pmatrix}
-178 & 276 \\
-137 & 234
\end{pmatrix}, 
$ where $A$ has eigenvalues 
$\la_1=20=2^2\cdot 5$, $\la_2=54=2\cdot 3^3$, $B$ has eigenvalues $\mu_1=-40=-2^3\cdot 5$, 
$\mu_2=96=2^4\cdot 3$.
We have that 
$$
A=SM\left(\begin{matrix}
\la_1 & 0 \\
0 & \la_2 
\end{matrix}
\right)M^{-1}S^{-1},\quad S=\left(\begin{matrix}
1 & 2 \\
1 & 1 
\end{matrix}
\right)\in\operatorname{GL}_2(\bZ),\quad
M=\left(\begin{matrix}
1 & 1 \\
0 & 34 
\end{matrix}
\right),
$$
$$
B=PN\left(\begin{matrix}
\mu_1 & 0 \\
0 & \mu_2 
\end{matrix}
\right)N^{-1}P^{-1},\quad P=\left(\begin{matrix}
2 & 1 \\
1 & 1 
\end{matrix}
\right)\in\operatorname{GL}_2(\bZ),\quad
N=\left(\begin{matrix}
1 & 1 \\
0 & 136 
\end{matrix}
\right).
$$
Thus, in the notation of Proposition \ref{pr:comp2}, 
$u=x=1$, $v=\la_2-\la_1=2\cdot 17$, $y=\mu_2-\mu_1=136=2^3\cdot 17$, 
$\nu_1=\pm 2^{\al_1}5^{\al_2}$,  $k=\pm 2^{\be_1}3^{\be_2}$, $\al_i,\be_i\in\bN\cup\{0\}$, $i=1,2$, 
and \eqref{eq:main2} becomes
$k\equiv 4\nu_1\,\text{mod}\, 136$. Solving the congruence, we get 
$\beta_1 \geq 3, \alpha_1 \geq 1$ or $\beta_1 = 2, \alpha_1 = 0$, and
\[
\beta_{2} - 2\beta_{1} - 5\alpha_{2} + 2\alpha_{1} \equiv t - 4 \pmod{16},
\]
where 
\[
t=
\begin{cases}
0, & \text{if } \operatorname{sgn}(k)=\operatorname{sgn}(\nu_{1}), \\[6pt]
8, & \text{otherwise}.
\end{cases}
\]
Finally, the set of all isomorphisms is $\Isom=\{B^jT(\nu_i)\}_{j\in\bZ}$, where $T(\nu_i)$ is given by \eqref{eq:T}, $\nu_2=0$.
\end{example}

\begin{example} 
Let \[
A=\begin{pmatrix}
99 & -89 \\
45 & -35
\end{pmatrix}
\] and let $B$ as in the previous example. 
Then $A$ has eigenvalues 
$\la_1=10=2\cdot 5$, $\la_2=54=2\cdot 3^3$.
We have that 
$$
A=SM\left(\begin{matrix}
\la_1 & 0 \\
0 & \la_2 
\end{matrix}
\right)M^{-1}S^{-1},\quad S=\left(\begin{matrix}
1 & 2 \\
1 & 1 
\end{matrix}
\right)\in\operatorname{GL}_2(\bZ),\quad
M=\left(\begin{matrix}
1 & 1 \\
0 & 44 
\end{matrix}
\right).
$$
Thus, in the notation of Proposition \ref{pr:comp2}, 
$u=x=1$, $v=\la_2-\la_1=44=4\cdot 11$, 
$\nu_1=\pm 2^{\al_1}5^{\al_2}$,  $k=\pm 2^{\be_1}3^{\be_2}$, $\al_i,\be_i\in\bN\cup\{0\}$, $i=1,2$, 
and  \eqref{eq:main2} becomes
$$
11k-34\nu_1\equiv 0\,\,\text{mod}\,\, 2^3\cdot 11\cdot 17,
$$ 
which implies that $k$ is divisible by $17$, which is impossible. Hence, $G_A$ is not isomorphic to $G_B$. 
\end{example}

\end{document}